\documentclass{article}

\usepackage[dvips]{graphicx}
\usepackage[ansinew]{inputenc}
\usepackage[margin=1in]{geometry}
\usepackage{graphbox}
\usepackage{float}
\usepackage{bbm}
\usepackage{makecell}
\usepackage{enumitem}

\usepackage{array,amsthm,amsmath,amssymb,mathrsfs,mdwlist,color,verbatim,multirow,amscd,rotate,stmaryrd,tikz,xcolor}

\usepackage{hyperref}
\usepackage[capitalize]{cleveref}
\usepackage{subcaption}

\usetikzlibrary{positioning}

\tikzset{%
  every neuron/.style={
    circle,
    draw,
    minimum size=1cm
  },
  neuron missing/.style={
    draw=none, 
    scale=4,
    text height=0.333cm,
    execute at begin node=\color{black}$\vdots$
  },
}

\usepackage[
backend=biber,
style=alphabetic,firstinits=true
]{biblatex}

\hypersetup{
    colorlinks=true,
    linkcolor=black,
    filecolor=magenta,      
    urlcolor=cyan,
}

\usetikzlibrary{cd}

\newcommand{\R}{\mathbb{R}}
\newcommand{\W}{\mathcal{W}}
\newcommand{\A}{\mathcal{A}}
\newcommand{\T}{\mathbb{T}}

\newcommand{\C}{\mathbb{C}}

\newcommand{\conv}{\text{conv}}

\newcommand{\V}{\mathcal{V}}

\newcommand{\rec}{\text{rec\,}}

\newcommand{\leg}{\langle}
\newcommand{\rig}{\rangle}

\newcommand{\vp}{\varphi}

\newcommand{\bs}{\backslash}

\newcommand{\mComp}{\texttt{mComp}}
\newcommand{\fComp}{\texttt{fComp}}

\newcommand{\reddot}{\tikz\draw[red,fill=red] (0,0) circle (.4ex);}

\newcommand{\bluedot}{\tikz\draw[blue,fill=blue] (0,0) circle (.4ex);}

\newcommand{\blackdot}{\tikz\draw[black,fill=black] (0,0) circle (.4ex);}

\newtheorem{thm}{Theorem}
\newtheorem{prop}[thm]{Proposition}
\newtheorem{cor}[thm]{Corollary}
\newtheorem{lemma}[thm]{Lemma}

\theoremstyle{definition}
\newtheorem{definition}{Definition}

\newtheorem{opq}{Open Question}

\newtheorem{rmk}{Remark}

\theoremstyle{definition}
\newtheorem{ex}{Example}

\newcommand{\jayden}[1]{\textcolor{teal!60!white}{[#1]}}

\title{Minimal Representations of Tropical Rational Functions}
\author{Ngoc M. Tran\thanks{Department of Mathematics, University of Texas-Austin, Austin, Texas, 78712 USA. Email: Ngoc M. Tran (ntran@math.utexas.edu), Jidong Wang (jidongw@utexas.edu)}
\and Jidong Wang\footnotemark[1]}
\date{}

\begin{document}

\maketitle

\begin{abstract}
{Tropical geometry sheds new light on analyzing classical statistical models of piecewise-linear nature. Representing a piecewise-linear function as a tropical rational function, many nontrivial results can be obtained. This paper studies the minimal ones of the above representation. We give two natural notions of complexity for tropical rational functions, the monomial complexity and the factorization complexity. We show that in dimension one, both notions coincide, but this is not true in higher dimensions. We give a canonical representation that is minimal for conewise linear functions on $\R^2$, which ties to the question of finding canonical representatives for virtual polytopes. We also give comparison bounds between the two notions of complexity. As a proof step, we obtain counting formulas and lower bounds for the number of regions in an arrangement of tropical hypersurfaces, giving a small extension for a result by Mont\'{u}far, Ren and Zhang. It also gives a lower bound on the number of vertices in a regular mixed subdivision of a Minkowski sum, giving a small extension for Adiprasito's Lower Bound Theorem for Minkowski sums. We also show that any piecewise-linear function is a linear combination of conewise linear function, which may have implications for model choice for multivariate linear spline regression.}

\end{abstract}

\section{Introduction}

Continuous piecewise-linear (PL) functions are an important class of functions featured in machine learning, statistics, and optimization. Training a deep ReLU network with a fixed architecture, for instance, can be understood as an optimization problem over a subset of such functions. Used as regression models, PL functions have good interpretability, while capturing global non-linearity. There are many ways to represent a PL function, such as max-min representation, nested absolute value representation, and tropical rational function representation, etc. \cite{ovchinnikov2002max,melzer1986expressibility,kripfganz1987piecewise,lin1994generalization,}. The tropical rational function (TRF) representation writes a PL function as a difference of two max-affine function, which is formally a rational function over the tropical algebra $\T$, hence the name. Recall that $\T=(\R,\oplus,\odot,\oslash)$ where $a\oplus b = \max\{a,b\},a\odot b = a+b$ and $a\oslash b=a-b$. Any convex PL function $g:\R^d\to \R$ can be written as a tropical polynomial, and any continuous PL function $\vp:\R^d\to \R$ can be written as a tropical rational function:
\begin{gather}
    g(x) = a_1\odot x^{\odot b_1}\oplus \cdots \oplus a_n\odot x^{\odot b_n}:=\max\{ \leg b_1,x \rig + a_1 ,\cdots, \leg b_n,x \rig + a_n\}, \\
    \vp(x)=g(x)\oslash h(x) :=\max\{ \leg b_1,x \rig + a_1 ,\cdots, \leg b_n,x \rig + a_n\} - \max\{ \leg d_1,x \rig + c_1 ,\cdots, \leg d_m,x \rig + c_m\},
\end{gather}
for some $a_i,c_i\in \R,b_i,d_i\in \R^d$. We emphasize that the exponents can be arbitrary real numbers.

We call $g\oslash h$ a \textit{representation} of $\vp$. The advantage of this representation is the connection with convex geometry and tropical geometry, which provides new perspectives for analyzing statistical models and makes the study of PL functions more systematic. There has been progress in several directions. First, the study on the complexity and expressivity of deep neural networks with piecewise-linear activation. The use of numbers of linear regions as a measure for network complexity appeared in \cite{montufar2014number,pascanu2013number}, whose results give an explanation for why deep networks are usually favorable than shallow networks. The expected number of linear regions, on the other hand, was used in \cite{hanin2019complexity} to explain why the expressivity expected from depth is not fully achieved. Other important results along this line include \cite{montufar2021sharp}, \cite{tseran2021expected}. Thinking of PL functions as tropical rational functions, the linear regions correspond to faces of certain polytopes, where combinatorial tools abound. Whether or not the debatable name ``tropical geometry" is mentioned, understanding neural networks from the associated Newton polytopes has been integrated into the standard toolbox for machine learning \cite{hertrich2022facets}, ever since the initial proposal by Zhang et al.\cite{zhang2018tropical} and Maragos et al.\cite{MCT21}. The second direction is regression analysis. The conventional way to implement piecewise-linear regression (a.k.a. linear spline model) is to introduce dummy variables, which brings two challenges. On one hand it makes it hard to build multivariate models; on the other, one needs to estimate where the non-linearity occurs. Tropical rational functions circumvent those difficulties since it provides a global expression for the model. The idea of fitting a tropical polynomial to data was proposed in \cite{MaTh20b}. Besides these two main directions, other applications of tropical geometry include \cite{aubin2022tropical,akian2021tropical}.

 Continuing this line of work, the goal of this paper is to explore the notion of a \emph{minimal} representation of a TRF. A PL function admits (infinitely) many TRF representations, as \Cref{ex:rational_function} shows. Seeking for a minimal representation is not only useful in statistics, but also interesting in polyhedral geometry. In statistics, a model with fewer parameters prevents overfitting. In polyhedral geometry, the question is equivalent to finding canonical representatives for \textit{virtual polytopes} (see \Cref{sec:monomial-comp}). 

 \begin{ex}\label{ex:rational_function}
The following function $\psi_1:\R^2\to \R$, is given by its linear pieces. The right hand side are two TRFs representing $\psi_1$.
\begin{equation}
   \psi_1(x,y)=\begin{cases}
x, &  0\leq x \leq y \\
y, &  0\leq y\leq x \\
0, &  \text{otherwise}
\end{cases},\quad \psi_1(x,y)=(x\oplus 0)\odot (y\oplus 0)\oslash (x\oplus y \oplus 0)=(x\odot y \oplus x\oplus y) \oslash (x\oplus y). 
\end{equation}
\end{ex}

Concretely, let $\vp:\R^d \to\R$ be a PL function. Let 
$$\mathcal{F}(\vp) = \{(g,h): \vp(x) = g(x) \oslash h(x) \quad \mbox{ for all } x \in \R^d\}$$ 
be the set of all TRFs that represent $\vp$. Let $\texttt{Comp}: g \mapsto \mathbb{N}$ be some function that measures the complexity of $g$. Then $\texttt{Comp}$ induces the following preorder on $\mathcal{F}(\vp)$ 
\[g_1\oslash h_1 \prec_{\texttt{Comp}} g_2\oslash h_2 \text{ if and only if }\texttt{Comp}(g_1)\leq \texttt{Comp}(g_2),\texttt{Comp}(h_1)\leq \texttt{Comp}(h_2). \]

We consider two natural notions of complexity for a tropical rational function, the monomial complexity $\mComp$ which is intrinsic to the function and the factorization complexity $\fComp$ which depends on the factorization. 

\begin{definition}\label{def:complexity}
Let $g$ be a tropical polynomial. Its monomial complexity $\mComp(g)$ is the number of linear regions $g$ has as a real-valued function. Given a nontrivial factorization $g = h_1 \odot \dots \odot h_m$ (see \Cref{sec:trop-poly}), its factorization complexity $\fComp(h_1, \dots, h_m)$ is $\sum_{i=1}^m \mComp(h_i)-(m-1)$.
\end{definition}

The reason for the name ``monomial'' complexity is because when one writes $g$ as a tropical sum of irredundant monomials, the number of linear regions equals the number of monomials. Note that the input of $\fComp$ is all the factors, so $\fComp$ depends on the factorization. The $m-1$ term in the definition of $\fComp$ is a normalizing factor, which we explain in \Cref{sec:notations}. Different irreducible factorizations can have different complexity.
Factorization of tropical polynomials is simple in dimension 1 but is NP-Hard and non-unique in dimension $\geq 2$ \cite{gao2001decomposition,tiwary2008hardness}. Therefore, one may suspect that theory of minimal representation is already quite rich in dimension $d = 2$. Our investigation supports this finding. When $d=1$, the two notions of complexity coincide and there is a unique minimal representation (\Cref{prop:minimal-unique-1-d}). However, when $d \geq 2$, factorization and monomial complexity can differ (\Cref{prop:two-notions-agree-in-1-d}). 
 
 Since our approach is geometric, we draw some analogy from classical algebraic geometry to explain how tropical geometry helps with understanding PL functions. A dictionary between the basic notions of these two fields is given in \Cref{tab:dictionary}. Algebraic geometry studies polynomial functions via their \textit{vanishing loci}. The interplay of the vanishing loci gives much information about the polynomial equations. Let $\V(F)$ be the vanishing locus of an ordinary polynomial $F$ on $\C^d$. Geometrically, this is an algebraic hypersurface. As a set, $\V(F)=\{x\in \C^d \mid F(x)=0\}$, while it admits more structure than a set. In particular, it comes with a multiplicity. For instance, $F(x)=x$ and $G(x)=x^2$ both vanish at $x=0$, but $F$ vanishes with multiplicity one while $G$ vanishes with multiplicity two. 
 
  Formal operations on hypersurfaces give a geometric incarnation of algebraic operations on polynomials. Consider ordinary polynomials $F(x,y)=x^2y$ and $G(x,y)=xy(y+1)$ defined on $\C^2$. Let $X$ be the line $\{y=0\}$, $Y$ be the line $\{x=0\}$, and $Z$ be the line $\{y+1=0\}$. Polynomial $F$ vanishes at $Y$ with multiplicity 2 and at $X$ with multiplicity 1, so we write $\V(F)=X+2Y$. Polynomial $G$ vanishes at $X,Y,Z$ all with multiplicity 1, so we write $\V(G)=X+Y+Z$. If we look at the ``vanishing locus'' of the rational function $\Phi=\frac{F}{G}=\frac{x}{y+1}$, we may say that it vanishes at $Y$ with multiplicity 1 and at $Z$ with multiplicity $-1$, meaning $\Phi$ has a \textit{pole} along $Z$ with multiplicity 1. By formally adding and subtracting the vanishing sets, we have
\begin{equation}
    \V(\Phi) = \V(F)-\V(G) = Y-Z
\end{equation}
The formal addition of the hypersurfaces correspond to multiplying the polynomials and the formal subtraction correspond to dividing the polynomials. In the above example, $X$ is canceled out. This corresponds to the cancellation of the $y$ factor in $F$ and $G$, and it means that $\Phi$ does not vanish at $X$. In general, $\V(\Psi)$ can be written as ``zeros $-$ poles''. We note two features of this correspondence.
\begin{enumerate}[label=\textbf{(F\arabic*)}]
    \item\label{itm:feature-1} For any rational function $\Phi$, $\V(\Phi)$ determines $\Phi$ up to a scalar multiple. 
    \item\label{itm:feature-2} The expression $\V(\Phi)=\text{zeros}-\text{poles}$ gives a representation $\frac{F}{G}$ of $\Phi$, where $\V(F)=\text{zeros}$ and $\V(G)=\text{poles}$. This representation is minimal in the sense that whenever $\Phi=\frac{\widetilde{F}}{\widetilde{G}}$, $F$ divides $\widetilde{F}$ and $G$ divides $\widetilde{G}$.
\end{enumerate}

\begin{figure}[H]
    \centering
    \begin{subfigure}{0.3\textwidth}
        \includegraphics[width=1.5in]{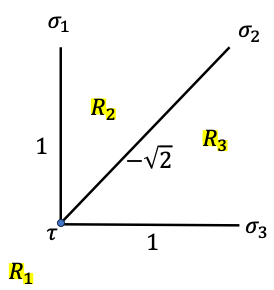}
    \caption{The bend locus $\V(\psi_1)$}
    \label{fig:bend-locus}
    \end{subfigure}
    \begin{subfigure}{0.3\textwidth}
        \includegraphics[width=1.5in]{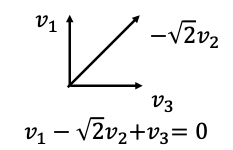}
        \vspace{0.25in}
    \caption{Balancing condition at $\tau$.}
    \label{fig:balancing-at-tau}
    \end{subfigure}
    \begin{subfigure}{0.3\textwidth}
        \includegraphics[width=2in]{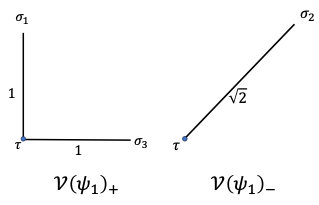}
    \caption{Positive and negative parts of $\V(\psi_1)$.}
    \label{fig:positive-and-negative}
    \end{subfigure}

    \begin{subfigure}{\textwidth}
    \centering
    \includegraphics[width=5in]{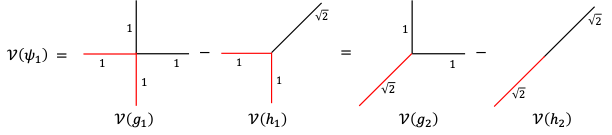}
    \caption{Two representations of $\psi_1$ induced by two different balancings. The subcomplexes in red indicate balancings.}
    \label{fig:two-reps}
    \end{subfigure}
    \caption{Geometric notions related to $\psi_1$ in \Cref{ex:rational_function}.}
\end{figure}

For a \textit{tropical} polynomial/rational function $\vp$ on $\R^d$, one can also appeal to its ``vanishing locus'', which we for convenience also denote as $\V(\vp)$. In the tropical world, ``vanishing loci'' are loci of non-linearity. As a set, it is
\begin{equation}\label{eq:def-of-nonlinear-locus}
    \V(\vp) = \{x\in \R^d\mid \vp \text{ is nonlinear at any neighborhood of }x\}.
\end{equation}

$\V(\vp)$ comes with the following extra structure.

\begin{itemize}
    \item It is a polyhedral complex with a canonical polyhedral structure, meaning that it consists of polyhedra which are glued together nicely along their facets. 
    \item It has pure dimension $d-1$, meaning that all the polyhedra maximal w.r.t. inclusion have dimension $d-1$.
    \item There is a weight function $w_\vp$ assigning a real number to each $(d-1)$-face, which encodes the degree of convexity of $\vp$ at that face. The weight is positive where $\vp$ is locally strictly convex and negative where $\vp$ is locally strictly concave.
    \item The weight function satisfies a \textit{balancing condition}, which reflects the rigidity of continuous PL functions.
\end{itemize}

While the formal definitions will be given in \Cref{sec:notations}, we illustrate the above notions using \Cref{ex:rational_function}. As is shown in \Cref{fig:bend-locus}, $\V(\psi_1)$ consists of three 1-faces $\sigma_1,\sigma_2$ and $\sigma_3$, and one 0-face $\tau$. The unit vectors that generate the rays $\sigma_1,\sigma_2$ and $\sigma_3$ are, respectively, $v_1=(0,1),v_2=(\frac{\sqrt{2}}{2},\frac{\sqrt{2}}{2})$ and $v_3=(1,0)$. $\V(\psi_1)$ divides $\R^2$ into three regions $R_1,R_2$ and $R_3$, where $\psi_1$ is linear. The 1-face $\sigma_1$ is the common facet of $R_1$ and $R_2$. From $R_1$ to $R_2$, the gradient change is $(1,0)-(0,0)=(1,0)$, which has length 1. Thus the weight on $\sigma_1$ is $w_{\psi_1}(\sigma_1)=1$. The gradient change from $R_2$ to $R_3$ is $(0,1)-(1,0)=(-1,1)$, which has length $\sqrt{2}$. However, $\psi_1$ is locally concave along $\sigma_2$, so we put the weight on $\sigma_2$ to be $w_{\psi_1}(\sigma_2)=-\sqrt{2}$. Similarly, the weight on $\sigma_3$ is $w_{\psi_1}(\sigma_3)=1$.  Observe that $w_{\psi_1}(\sigma_1)v_1+w_{\psi_1}(\sigma_2)v_2+w_{\psi_1}(\sigma_3)v_3=0$. We say that $\V(\psi_1)$ is \textit{balanced} at $\tau$ (\Cref{fig:balancing-at-tau}). The weight function is analogous to the order of vanishing. It is part of the information $\V(\psi_1)$ has. For any TRF $\vp$, the underlying set of $\V(\vp)$ is called the \textit{support} of $\V(\vp)$. Consider another function $\psi_2=\psi_1^{\odot\frac{3}{2}}$. Then $\V(\psi_2)$ and $\V(\psi_1)$ have the same support, but $\V(\psi_1)\neq \V(\psi_2)$. Since $w_{\psi_2}=\frac{3}{2}w_{\psi_1}$, one may write 
\begin{equation}
    \V(\psi_2)=\frac{3}{2}\V(\psi_1)
\end{equation}

For a tropical polynomial $f$, $\V(f)$ is called a \textit{tropical hypersurface}. As in the case of algebraic hypersurfaces, one can formally add and subtract tropical hypersurfaces and more generally bend loci of TRFs, which also correspond to operations on tropical polynomials/TRFs (See \Cref{prop:equivalance}). In \Cref{ex:rational_function}, from the two representations $g_1\oslash h_1$ and $g_2\oslash h_2$ of $\psi_1$, one writes 
\begin{equation}
    \V(\psi_1) = \V(g_1) - \V(h_1) = \V(g_2) - \V(h_2).
\end{equation}
Therefore, geometrically, we think of a representation of $\psi_1$ as a difference of tropical hypersurfaces. Recall the two features \ref{itm:feature-1} and $\ref{itm:feature-2}$ of $\V(\Phi)$ for an \textit{ordinary} rational function $\Phi$. Analogous to \ref{itm:feature-1}, in the tropical world, $\V(\vp)$ determines $\vp$ up to a linear function. However, property \ref{itm:feature-2} fails in dimension $\geq 2$. To make this precise, we make the following definition.
\begin{table}[]
\footnotesize
    \centering
    \begin{tabular}{|c|c|c|c|c|c|c|}
    \hline
      \makecell{ \textbf{classical} \\ \textbf{concept} } & \makecell{addition \& \\ subtraction} & multiplication/division & \makecell{zero/pole} & \makecell{ordinary \\ polynomial}  & \makecell{rational\\ functions} & \makecell{order of \\ vanishing}  \\ \hline
      \makecell{\textbf{tropical} \\ \textbf{counterpart}}  & \makecell{taking \\ maximum} & usual addition/subtraction & \makecell{locus of strict\\ convexity/concavity} & \makecell{convex \\ PL function}  & PL functions & weight \\\hline
    \end{tabular}
    \caption{A dictionary between classical algebraic geometry and tropical geometry}
    \label{tab:dictionary}
\end{table}

 \begin{definition}
    For any TRF $\vp$, define the \textit{positive part} $\V(\vp)_+$ and the \textit{negative part} $\V(\vp)_-$ of $\V(\vp)$ by \begin{equation}
    \begin{split}
        \V(\vp)_+ = \text{the closure in }\R^d\text{ of }\{x\in \R^d\mid \vp\text{ is locally strictly convex in some neighborhood of }x\},  \\
        \V(\vp)_- = \text{the closure in }\R^d\text{ of }\{x\in \R^d\mid \vp\text{ is locally strictly concave in some neighborhood of }x\}. 
    \end{split}
\end{equation}
Both $\V(\vp)_+$ and $\V(\vp)_-$ inherit the weight $|w_\vp|$ from $\V(\vp)$.
\end{definition}

$\V(\vp)_+$ and $\V(\vp)_-$ correspond to the subcomplexes of $\V(\vp)$ with positive and negative weights, respectively, hence the names. In other words, $\V(\vp)_+$ encodes where $\vp$ is locally strictly convex and the degree of convexity, while $\V(\vp)_-$ encodes where $\vp$ is locally strictly concave and the degree of concavity. (There are saddle points where $\vp$ is neither convex nor concave. The closure in the above definition is meant to include those points. See \Cref{rmk:convept-codim-one}). They are the analogues to zeros and poles in the case of ordinary rational functions. Unlike the case for ordinary rational functions, $\V(\vp)_+$ or $\V(\vp)_-$ may not be tropical hypersurfaces. Indeed, in \Cref{ex:rational_function}, $\V(\psi_1)_+$ consists of $\sigma_1,\sigma_3$ and $\tau$. $\V(\psi_1)_-$ consists of $\sigma_2$ and $\tau$ (\Cref{fig:positive-and-negative}). They are not balanced, thus not tropical hypersurfaces. Therefore, one cannot read off a representation of $\psi_1$ from $\V(\psi_1)_+$ or $\V(\psi_1)_-$, let alone a minimal one. To construct a representation given the information of $\V(\vp)_+$, the key observation is the following. For any representation $\vp=g\oslash h$ with tropical hypersurfaces $X=\V(g)$ and $Y=\V(h)$,
\begin{enumerate}[label=\textbf{(\Alph*)}]
\item\label{itm:one} $\V(\vp)_+$ is a subcomplex of $X$; 
$\V(\vp)_-$ is a subcomplex of $Y$.
\item\label{itm:two}  Let $w_X$ be the weight function on $X$ and $w_\vp$ the weight function on $\V(\vp)$. On each $(d-1)$-face $\sigma$ of $\V(\vp)_+$, $w_X(\sigma)\geq w_\vp(\sigma)$. In other words, $X$ is locally more convex than $\vp$. Likewise, $-Y$ is locally more concave than $\vp$.
\end{enumerate}

In other words, to find a representation for $\vp$, one needs to balance the possibly unbalanced complexes $\V(\vp)_+$ and $\V(\vp)_-$ by adding more facets with appropriate weights. This motivates the following definition.

\begin{definition}\label{def:balancing}
    Let $X$ be a tropical hypersurface and $\vp$ be a TRF. If $X$ has the properties specified in \ref{itm:one} and \ref{itm:two} relative to $\V(\vp)_+$, then $X$ is a \textit{balancing} for $\V(\vp)_+$.
    \end{definition}
    
    We will see that as one constructs a balancing for $\V(\vp)_+$, one also balances $\V(\vp)_-$ (\Cref{lem:balancing-gives-rep}). Therefore, a balancing induces a representation.
    We describe a natural balancing for any \textit{conewise} linear function $\vp:\R^2\to \R$. By definition, $\V(\vp)_+$ consists of rays $\gamma_1,...,\gamma_n$ originating from the same point. This data is the same as a collection of vectors $v_1,...,v_n$, such that $v_i$ generates $\gamma_i$ and $||v_i||$ is the weight of $\gamma_i$. Let $v_{n+1}=-\sum_{i=1}^nv_i$. $\V(\vp)_+$ is balanced if and only if $v_{n+1}=0$. If $v_{n+1}\neq 0$, let $\gamma_{n+1}$ be the ray generated by $v_{n+1}$ with weight $||v_{n+1}||$ and $X$ be the complex consisting of $\V(\vp)_+$ and $\gamma_{n+1}$ with weights inherited from both. By definition, $\sum_{i=1}^{n+1}v_i=0$, so $X$ is a balancing for $\V(\vp)_+$. Let $Y$ be the complex consisting of $\gamma_{n+1}$ and $\V(\vp)_-$. A simple calculation shows that $Y$ is also balanced. We call $X$ the \textit{canonical balancing} for $\V(\vp)_+$ and the representation $\V(\vp)=X-Y$ the \textit{canonical representation} for $\vp$. We explain why it deserves the name ``canonical'' in \Cref{sec:monomial-comp}. Intuitively, this is the representation that ``uses the least effort''. Our first result says the canonical representation is minimal w.r.t. the monomial complexity.
    
\begin{prop}\label{prop:fan-unique}
    Let $\vp:\R^2\to\R$ be a conewise linear function. The canonical representation of $\vp$ is minimal w.r.t. $\mComp$.
    \end{prop}

In \Cref{ex:rational_function}, the representation $g_2\oslash h_2$ in \Cref{fig:two-reps} is the canonical representation of $\psi_1$. We give a possible generalization of canonical representations to conewise linear functions in higher dimensions in \Cref{sec:monomial-comp}. For a tropical polynomial $g$, there is a nice correspondence between the function itself, the tropical hypersurface $\V(g)$ and its lifted Newton polytope $\Delta^\uparrow(g)$ (see \Cref{prop:equivalance}). Therefore, notions of complexity migrate to tropical hypersurfaces as well as polytopes. One can thus talk about \textit{minimal balancing problem} (see \Cref{sec:minimal-rep}). Besides serving as a step stone for studying the minimal representation problem, the minimal balancing problem turns out interesting on itself. One question in machine learning is if one can recover the architecture and parameters of a neural network by simply querying the network \cite{rolnick2020reverse,oh2019towards}. The minimal balancing problem, as a comparison, asks if one can construct a simplest network structure that is compatible with the queries. Another example is the Inverse Voronoi Problem \cite{aloupis2013fitting}, which asks how to fit a Voronoi diagram to a subdivision of the plane, where the subdivision is not necessarily convex. In that context, the minimal balancing problem asks for the Voronoi diagram with the fewest sites. 
    
    We study the difference between monomial complexity and factorization complexity. \Cref{prop:two-notions-agree-in-1-d} says $\mComp$ and $\fComp$ are the same when $d=1$, whereas for $d\geq 2$, they have distinct behavior, and the minimal balancing problem w.r.t. the monomial complexity and w.r.t. the factorization complexity can be different.  \Cref{thm:lower-bound} gives comparison bounds between the monomial complexity and the factorization complexity. Under genericity conditions, the factorization complexity of a tropical polynomial is much smaller than its monomial complexity.
   
\begin{thm}\label{thm:lower-bound}
Let $g_1,...,g_m: \R^d\to \R$ be tropical polynomials with generic parameters. Then
\[\mComp(g_1\odot \cdots \odot g_m) \geq  \fComp(g_1, \cdots, g_m)+\sum_{k=2}^d\binom{m}{k}\]
The equality holds if and only if all the intersections of $\V(g_i)$'s are affine subspaces, i.e. if and only if all $g_i$'s are binomials.
\end{thm}

Instances of the above phenomenon were already observed in \cite{magnani2009convex} in the context of PL curve fitting and it has the implication on how PL functions can be efficiently stored.
When all $g_i$'s are binomials, \Cref{thm:lower-bound} recovers the classic region counting formula for arrangements of real affine hyperplanes \cite[Page 1]{orlik2013arrangements}. As part of the proof, we derive a counting formula for the number of regions cut out by an arrangement of tropical hypersurfaces, or equivalently, the number of vertices in the regular mixed subdivision of a Minkowski sum, an important topic in polyhedral geometry \cite{gritzmann1993minkowski,weibel2012maximal,fukuda2004zonotope,fukuda2007f}. Using a series of theories developed in \cite{adiprasito2021relative} and \cite{adiprasito2021partition} plus the Lefschetz theorem for simplicial polytopes, Adiprasito gave a lower bound for the number of vertices of Minkowski sums \cite{adiprasito2017lefschetz}. In comparison, \Cref{thm:lower-bound} also applies to regular mixed subdivisions of Minkowski sums. We recover part of his result as a special case (\Cref{thm:lower-bound-polytope}), as well as give a lower bound that is  asymptotically slightly stronger. Our proof is inspired by another result by Adiprasito on the topology of intersections of tropical hypersurfaces \cite{adiprasito2020note}. Using the lower bound in \cite{adiprasito2017lefschetz} and a result in \cite{montufar2021sharp}, Tseran and Mont\'{u}far derived a lower bound for the number of linear regions of maxout networks \cite[Theorem 8]{tseran2021expected}. Our proof thus provides another path to their result. \Cref{thm:lower-bound} is also complementary to Theorem 3.7 in \cite{montufar2021sharp}, an upper bound for the number of regions in arrangements of tropical hypersurfaces.

Our primary investigation suggests that factorization complexity is more suitable for statistical consideration, because it is more compatible with the process of building global statistical models by combining local models. This is illustrated by \Cref{prop:properties-of-factorization-len} and \Cref{prop:generic-fans}. The following theorem, which doesn't seem obvious without tropical tools, provides the foundation for combining local models to get global models.

\begin{thm}\label{thm:piecewise-conewise}
    Every piecewise-linear function is a linear combination of conewise linear functions, among which distinct conewise linear functions may have their cones originate from distinct points.
\end{thm}

Finally, we obtain the following result concerning computing a good balancing/representation if not necessarily a minimal one. For any TRF $\vp:\R^d\to \R$, $X=\V(\vp)_+$ has another natural balancing. One can simply take all the affine hyperplanes spanned by all its $(d-1)$-faces, which we call the \textit{associated arrangement} of $X$, denoted $\A_X$ (see \Cref{fig:canonical-balancing}). With respect to the factorization complexity, the following result shows that in dimension 2, the associated arrangement has small, if not minimal, complexity.

\begin{prop}\label{prop:lower-bound-balancing}
Let $X=\V(\vp)_+$ for some PL function $\vp: \R^2\to \R$. Let $\V_1+\cdots +\V_n$ be its minimal balancing w.r.t. the factorization complexity. Then 
\[\fComp(\A_X)\leq 3\fComp(\V_1+\cdots +\V_n)\]
where $\A_X$ is given as the sum of all the hyperplanes. 
\end{prop}

\textbf{Organization.} In \Cref{sec:notations}, we introduce notations and basic facts that will be used in this paper. In \Cref{sec:complexity-of-functions}, we
prove \Cref{thm:lower-bound}. In \Cref{sec:minimal-rep}, we prove \Cref{prop:fan-unique}, \Cref{prop:lower-bound-balancing} and \Cref{thm:piecewise-conewise}. We conclude with open questions in \Cref{sec:conclusion}.

 \textbf{Acknowledgement.} We thank Karim Adiprasito for explaining his results in \cite{adiprasito2020note} and Sam Payne for helpful discussions. We thank the referees for constructive comments. Ngoc Tran and Jidong Wang are supported by NSF Grant DMS-2113468 and the NSF IFML 2019844 award to the University of Texas at Austin.

\section{Notations and preliminaries}\label{sec:notations}
\subsection{Tropical polynomials, hypersurfaces, and Newton polytopes}\label{sec:trop-poly}

Recall the definition of tropical polynomials, tropical rational functions, and tropical hypersurfaces from the introduction. Let $g$ be a tropical polynomial on $\R^d$. A \textit{factorization} of $g$ is a representation as a product $g=g_1\odot \cdots\odot g_m$. A factorization is called \textit{nontrivial} if none of the factors is a power of $g$. This excludes cases like $g=g^{\odot \frac{1}{2}}\odot g^{\odot \frac{1}{2}}$. All factorizations discussed in this paper will be nontrivial. Let $g$ be given as a tropical sum of monomials.
We call $g$ \textit{reduced} if removing any of the monomials in $g$ will change the underlying real-valued function. In other words, $g$ is reduced if none of its monomials is redundant. For instance, $g(x,y)=x^{\odot 2}\oplus x\odot y\oplus y^{\odot 2}$ is not reduced, since removing $x\odot y$ doesn't change the real-valued function. A factorization is reduced if all the factors are reduced.

A tropical hypersurface $\V(g)$ includes the following information.
\begin{itemize}
    \item A canonical polyhedral structure on $\V(g)$: note that each linear region of $g$ is a $d$-dimensional polytope. The faces of $\V(g)$ are the faces of these polytopes of dimension lower than $d$.
    \item A weight function $w_g$: each $(d-1)$-face $\sigma$ of $V(g)$ is the common facet of two linear regions, say, $R_1$ and $R_2$, of $g$. Suppose the gradient of $g$ is $b_1$ on $R_1$ and $b_2$ on $R_2$. Then $w_g(\sigma)=||b_1-b_2||$, the 2-norm of $b_1-b_2$.
    \item The balancing condition: This is a local condition at every $(d-2)$-face of $\V(g)$. Pick a $(d-2)$-face $\tau$. Suppose it's the common facet of $(d-1)$-faces $\sigma_1,...,\sigma_n$. Let $\R \tau$ be the affine subspace spanned by $\tau$. Let $\R_+\sigma_i$ be affine half subspace spanned by $\sigma_i$ with boundary $\R\tau$. Translate everything so that $\R\tau$ contains the origin. Now $\R^d/\R\tau\cong \R^2$ and $\R_+\sigma_i/\R\tau$ can be identified with a ray. Let $v_i$ be the generator of the ray $\R_+\sigma_i/\R\tau$. The balancing condition is
\begin{equation}\label{eq:balancing-condition}
  \sum_{i}w_g(\sigma_i)v_i=0 . 
\end{equation}

\end{itemize}

 Consider a collection of tropical hypersurfaces $X_1,...,X_n$ that are \textit{generic}, meaning that the coefficients of tropical polynomials $g_1,...,g_n$ come from a probability distribution with continuous density. Their intersection will be transverse: the intersection of a face of codimension $k$ and a face of codimension $l$ is either empty or a face of codimension $k+l$. In this case $X_1\cap\cdots\cap X_n$ is a \textit{complete intersection}. As described above, a complete intersection has a canonical polyhedral structure given by intersections of the faces of $X_i$'s. Examples of polyhedral structures of tropical hypersurfaces in $\R^3$ and a complete intersection are given in \Cref{fig:trop-R3-and-intersection}.

 \begin{figure}
     \centering
     \begin{subfigure}{0.4\textwidth}
     \centering
         \includegraphics[width=2in]{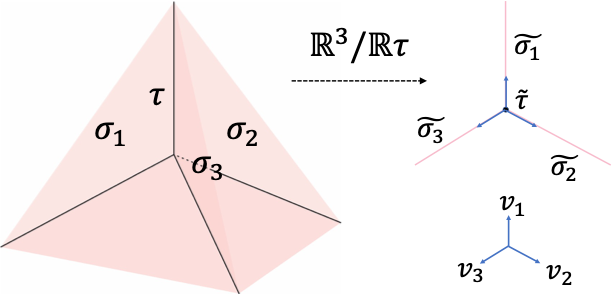}
     \caption{$\V(g)$ and the balancing condition at $\tau$. The rays $\widetilde{\sigma_1},\widetilde{\sigma_2},\widetilde{\sigma_3}$ are images of $\R_+\sigma_1,\R_+\sigma_2,\R_+\sigma_3$, respectively, in $\R^3/\R\tau$. The weighted sum of their respective generating vectors $v_1,v_2,v_3$ is zero.}
     \label{fig:trop-R3}
     \end{subfigure}
     \begin{subfigure}{0.4\textwidth}
     \centering
         \includegraphics[width=1in]{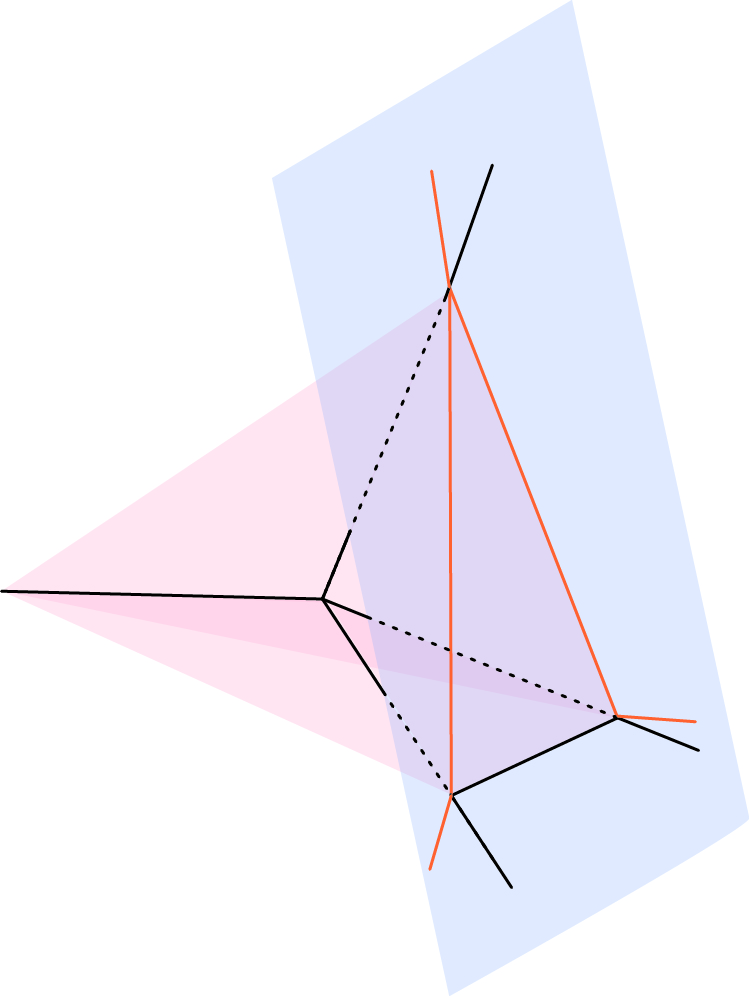}
         \caption{Intersecting $\V(g)$ with a hyperplane. The intersection is colored in orange.}\label{fig:intersection}
     \end{subfigure}
     \caption{ Here, $g=x^{\odot -1}\odot z^{\odot -1}\oplus x^{\odot 1}\odot z^{\odot -1}\oplus  y^{\odot 1}\odot z^{\odot 1}\oplus y^{\odot 2}$. $\V(g)$ is the normal fan of a tetrahedron, consisting of one 0-face, four 1-faces and six 2-faces.}\label{fig:trop-R3-and-intersection}
 \end{figure}

\begin{rmk}\label{rmk:meaning-of-weight}
    Experts in tropical geometry may notice the departure of our definitions of the weight function and the balancing condition from the usual ones. We consider arbitrary polytopes instead of just rational polytopes. For rational polytopes, the weight function and the balancing condition considered here is equivalent to the usual ones.
\end{rmk}

The Euler characteristic $\chi(X)$ of a polyhedral complex $X$ is the alternating sum
\begin{equation}
    \sum_{i=0}^d (-1)^i \text{ number of $i$-faces.}
\end{equation}
Recall that $\chi(X)$ doesn't depend on the cell structure\footnote{In this paper, the term `cell' and `face' are mostly interchangeable. More rigorously, `cell' only concerns the topology, whereas `face' is a more appropriate term when the polyhedral structure is emphasized.}, and it satisfies the inclusion-exclusion principle.

\begin{lemma}\label{lem:inclusion-exclusion} \cite[p246]{ziegler2012lectures}
Let $X,Y\subset \R^d$ be two polyhedral complexes such that $X\cap Y$ is also a polyhedral complex. Then
\begin{equation}
    \chi(X\cup Y)=\chi(X)+\chi(Y)-\chi(X\cap  Y).
\end{equation}
\end{lemma}

The \textit{Newton polytope} $\Delta(g)$ and the \textit{lifted Newton polytope} $\Delta^\uparrow(g)$ of a tropical polynomial $g=b_1\odot x^{\odot a_1}\oplus \cdots \oplus b_n\odot x^{\odot a_n}$ are, respectively, 
\begin{equation}
    \begin{split}
       &\Delta(g):=\conv\{a_1,a_2,...,a_n\}\subset \R^d,  \\
    &\Delta^\uparrow(g):=\conv((a_1,b_1),(a_2,b_2),...(a_n,b_n)\}\subset \R^{d+1}.
    \end{split}
\end{equation}
The lifted Newton polytope $\Delta^\uparrow(g)$ induces a subdivision on $\Delta(g)$ by projecting the upper faces of $\Delta^\uparrow(g)$ down onto $\Delta(g)$ from above. Subdivisions of polytopes obtained this way are called \textit{regular}. This subdivision is dual to the subdivision of $\R^d$ by $\V(g)$ in a dimension-reversing manner, as is shown in \Cref{fig:lifted-newton-poly}.

\begin{figure}[H]
    \centering
    \begin{subfigure}{0.22\textwidth}
       \includegraphics[width=1.2in]{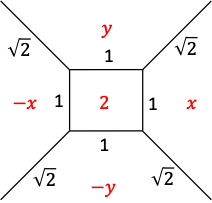} 
       \caption{$\V(g)$ with weights labeled on each 1-face.}
    \end{subfigure}\; \begin{subfigure}{0.22\textwidth}
       \includegraphics[width=1.4in]{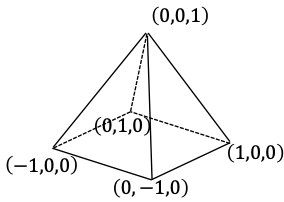}
       \caption{The lifted Newton polytope $\Delta^\uparrow(g)$.}
    \end{subfigure} 
    \;\begin{subfigure}{0.22\textwidth}
       \includegraphics[width=1.2in]{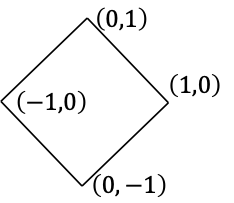}
       \caption{The Newton polytope $\Delta(g)$.}
    \end{subfigure}
    \;\begin{subfigure}{0.22\textwidth}
       \includegraphics[width=1.1in]{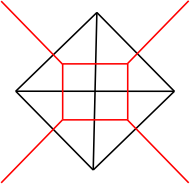}
       \caption{The subdivision on $\Delta(g)$ and the dual subdivision on $\R^2$.}
    \end{subfigure}
    \caption{Geometric objects associated with $g(x,y)=\max\{x,y,-x,-y,2\}$.}
    \label{fig:lifted-newton-poly}
\end{figure}

 A special family of tropical polynomials is those representing \textit{conewise linear functions}. For those tropical polynomials, the regular subdivision is just the original polytope. The corresponding tropical hypersurfaces are \textit{polyhedral fans}, which are the normal fans of their Newton polytopes. For the definitions of these concepts, see \cite{ziegler2012lectures}. We summarize some useful properties of the above 
constructions.

\begin{prop}\label{prop:equivalance}
    For reduced tropical polynomials $g,h:\R^d\to \R$, one has the following:
    \begin{enumerate}
        \item[(a)] $\V(g)=\V(h)$ as weighted polyhedral complexes if and only if $g\oslash h$ is a linear function. In other words, $\V(g)$ determines $g$ up to a linear function.
        \item[(b)] 
        $g$ and $h$ are the same real-valued function if and only if $\Delta^\uparrow(g)=\Delta^\uparrow(h)$.

        \item[(c)] For a pure $(d-1)$-dimensional polyhedral complex $X\subset \R^d$ with weight function $w_X$, there is some tropical polynomial $g:\R^d\to \R$ such that $\V(g)=X$ if and only if $w_X$ is nonnegative and $X$ is balanced. 
    \end{enumerate}
\end{prop}

\begin{proof}
Since the bend locus of $g\oslash h$ is $\V(g)-\V(h)$ (see \Cref{sec:TRF}), $g$ and $h$ differ by a linear function if and only if $\V(g)=\V(h)$, so (a) is true. Since $g$ is reduced, the monomials in $g$ correspond bijectively to the vertices of $\Delta^\uparrow(g)$. In other words, $\Delta^\uparrow(g)$ determines $g$ uniquely as a real-valued function, so (b) is true. When all the coefficients of $g$ are rational, (c) is Proposition 3.3.10 in Section 3.3 of \cite{maclagan2015introduction}, whose proof works verbatim for arbitrary coefficients.
\end{proof}

\subsection{Operations on tropical hypersurfaces}\label{sec:operations-on-hyper}

The operations on tropical hypersurfaces will be essential throughout the whole paper, so we spend this section clarifying how addition, subtraction, and scalar multiplication work for tropical hypersurfaces. A totally rigorous treatment of the ideas can be done using indicator functions, but that would involve excessive technicality that is more distracting than inspiring. Therefore, we give the following description with plenty of examples that cover all possible situations, which should suffice for the purpose of this paper.

\begin{rmk}\label{rmk:convept-codim-one}
     We emphasize that in the above description, transparency/light emission only makes sense in dimension $d-1$. In other words, the weight function in our context is a concept only concerning the $(d-1)$-dimensional part of a complex. This can be thought of as an analogous phenomenon from classical algebraic geometry: for instance, the function $\Phi(z_1,z_2)=\frac{z_1}{z_2}$ on $\C^2$ isn't defined at $(0,0)$, but one still says that $\Phi$ vanishes along $\{z_1=0\}$. Roughly speaking, modifying a tropical hypersurface in codimension larger than one won't cause problems within the scope of this paper. For example, one can refine the polyhedral structure of a tropical hypersurface by subdividing any $(d-1)$-face, but the refined polyhedral complex still determines the same function. 
\end{rmk}

A pure $(d-1)$-dimensional weighted polyhedral complex can be thought of as an artifact made by putting pieces of the above glass together. In particular, for any TRF $\vp$, $\V(\vp)$, $\V(\vp)_+$ and $\V(\vp)_-$ are all such artifacts, where the weight function $w_\vp$ specifies the level of non-transparency on each $(d-1)$-face. For $X$ and $Y$ two pure $(d-1)$-dimensional polyhedral complexes with weight functions $w_X$ and $w_Y$, respectively, and $\alpha>0$, 
\begin{itemize}
    \item $\alpha X$ simply amplifies the non-transparency/light-emission of $X$ by $w_{\alpha X}=\alpha w_X$;
    \item negation $-X$ reverses non-transparency and light-emission by $w_{-X}=-w_X$;
    \item $X+Y$ is the superposition of $X$ and $Y$ and $X-Y=X+(-Y)$.
\end{itemize}

The above description applies for weighted pure $(d-1)$-dimensional polyhedral complexes in general, but we are only concerned with subcomplexes of $\V(\vp)$ for some TRF $\vp$. For examples of addition and subtraction, see \Cref{fig:adding-tropical-hypersurfaces,fig:subtraction}. The addition is commutative and associative, and it's easy to check that if $X$ and $Y$ are balanced, then $\alpha X$ and $X\pm Y$ are also balanced. In addition to that, the operators $\V$, $\Delta$ and $\Delta^\uparrow$ on tropical polynomials satisfy the following.
\begin{equation}\label{eqn:trop-poly-operation-correspondence}
     \begin{split}
         \V(g^{\odot\alpha}\odot h^{\odot \beta}) = \alpha\V(g) + \beta\V(h), \\
    \Delta(g^{\odot\alpha}\odot h^{\odot \beta}) = \alpha\Delta(g) + \beta\Delta(h), \\
    \Delta^\uparrow(g^{\odot\alpha}\odot h^{\odot \beta}) = \alpha\Delta^\uparrow(g) + \beta\Delta^\uparrow(h),
     \end{split}
 \end{equation}
where the sum between polytopes is Minkowski sum.
\begin{figure}
    \centering
    \includegraphics[width=4in]{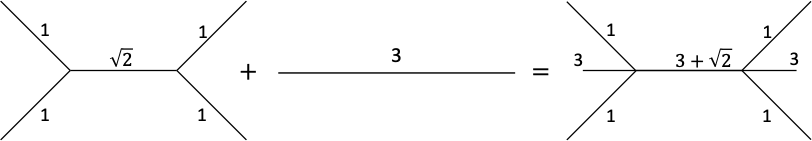}
\includegraphics[width=3.5in]{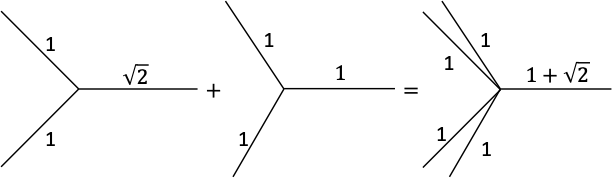}
    \caption{Two examples of adding tropical hypersurfaces with weights labeled on each 1-cell.}
    \label{fig:adding-tropical-hypersurfaces}
\end{figure}

Let $g$ be a reduced tropical polynomial. Recall that its \textit{monomial complexity} $\mComp(g)$, is the number of linear regions as a PL function. For a reduced factorization $g=h_1\odot \cdots \odot h_m$, its \textit{factorization complexity} is

\[\fComp(h_1,\cdots, h_m) = \sum_{i=1}^m \mComp(h_i)-(m-1).\]

The $-(m-1)$ term comes from normalization. For any tropical polynomial $G(x)=b_1\odot x^{\odot a_1}\oplus \cdots \oplus b_n\odot x^{\odot a_n}$, one can factor out the first term and write
\begin{equation}
    G(x)=b_1\odot x^{\odot a_1}\odot (0 \oplus (b_2-b_1)\odot x^{\odot (a_2-a_1)} \oplus \cdots \oplus (b_n-b_1)\odot x^{\odot (a_n-a_1)}) = b_1\odot x^{\odot a_1} \odot \widetilde{G}(x).
\end{equation}
Note that $\widetilde{G}(x)$ is defined by one fewer monomial than $G$. One can do this for each of the factors $h_i$ of $g$ and get $h_i=r_i\odot \widetilde{h_i}$ for $r_i$ some monomial and $\widetilde{h_i}$ a polynomial defined by one fewer monomial than $h_i$. Now
\begin{equation}\label{eq:normalization}
    g = r_1\odot\cdots\odot r_m \odot (\widetilde{h_1}\odot \cdots \odot \widetilde{h_m}).
\end{equation}
Note that $r_1\odot \cdots \odot r_m$ is a single monomial. To define $g$, one needs to specify the monomial $r_1\odot \cdots \odot r_m$ and the monomials in $\widetilde{h_i}$. Since each $\widetilde{h_i}$ takes one fewer monomial to define than $h_i$, the normalization \eqref{eq:normalization} requires $-m+1$ fewer monomials.

    Due to \Cref{prop:equivalance}, notions of complexity apply equally well to polytopes and tropical hypersurfaces, each carrying different meanings (see \Cref{tab:meaning-of-complexity}). The corresponding notions of complexity for polytopes are straightforward.  Through this correspondence, $\mComp(P)$ for a polytope is the number of vertices $P$ has, and $\fComp(P_1,P_2,\cdots,P_m)$ is the sum of the number of vertices of $P_i$'s minus $m-1$. Here, the $-(m-1)$ comes from the fact that, to determine a polytope $P$ from its factors $P_1,...,P_m$, one can assume $P_1,...,P_m$ each has a vertex at the origin. Each $P_i$ is determined by $\mComp(P_i)-1$ vertices. This determines $P$ up to a translation. Therefore, to determine $P$ from $P_1,...,P_m$, one needs 
    \begin{equation}
     \sum_{i=1}^m(\mComp(P_i)-1)+1=\sum_{i=1}^m\mComp(P_i)-(m-1)
    \end{equation}
    vertices.

\begin{table}[h]
        \centering
        \begin{tabular}{c|c|c}
        $\mComp(g) $ & $\mComp(P)$ & \mComp(X) \\ \hline
         $\#$ of monomials & $\#$ of vertices & $\#$ of regions in $\R^d\bs X$
    \end{tabular}
        \caption{Meanings of complexity for tropical polynomials, polytopes, and tropical hypersurfaces}
        \label{tab:meaning-of-complexity}
    \end{table}

\begin{rmk}
    The difference between $\mComp(P)$ and $\fComp(P_1,\cdots , P_m)$ might be an interesting measure. On one hand, fixing a polytope $P$ and varying the factorization $P=P_1+\cdots + P_m$, this difference
    can be regarded as a measure of the efficiency of the decomposition. On the other hand, fixing the combinatorial types of $P_1,...,P_m$ and changing their geometric realizations, the difference $\mComp(P_1+\cdots+P_m)-\fComp(P_1,...,P_m)$ can be regarded as a measure of the transversality of the Minkowski sum. For instance, if $P=\sum_{k=1}^m I_k\subset \R^d$ is a full dimensional zonotope for generic line segments $I_1,...,I_m$, then
    \[\mComp(P)=\sum_{k=0}^d\binom{m}{k}, \quad \fComp(I_1+\cdots + I_m)=m+1.\]
    As some of the line segments move and become coplanar, the above difference becomes smaller. A zonotope is dual to a central hyperplane arrangement. We don't know if this number is a matroidal invariant or not.
    \end{rmk}
    
\begin{ex}
    Consider (a) the $d$-simplex $\Delta_d$ and (b) the $d$-cube with the factorization into $d$ line segments. Since the $d$-simplex is indecomposable, its factorization complexity and monomial complexity are both $d+1$. The $d$-cube has monomial complexity $2^d$. However, the $d$-cube factors into $d$-line segments, and the complexity of this factorization is $2d-(d-1)=d+1$, the same as the $d$-simplex. This is consistent with the following intuition. A $d$-simplex is defined by fixing $d$-vertices, while a $d$-cube is defined by fixing one vertex and connecting $d$-line segments. They contain the same amount of information.
\end{ex}

We note an essential topological property of tropical hypersurfaces. 

\begin{lemma}[Euler-Poincar\'{e} relation] The Euler characteristic of a tropical hypersurface $X$ satisfies
\begin{equation}
    \chi(X)+(-1)^d\mComp(X)=(-1)^d.
\end{equation}
 \end{lemma}

 \begin{proof}
     This follows from the usual Euler-Poincar\'{e} relation for polytopes \cite[p231]{ziegler2012lectures}, the fact that $X$ is dual to a polytope, and the definition of $\mComp$.
 \end{proof}

 Finally, we mention two notions associated with any intersection of generic tropical hypersurfaces that will be used later. Let $X\subset \R^d$ be such an intersection and $\sigma$ be any of its faces. The \textit{recession fan} $\rec(X)$ of $X$ is what $X$ looks like ``from a distance", and the \textit{tangent fan} $T_\tau(X)$ of $X$ at $\tau$ is what $X$ looks like locally at $\tau$. More precisely, 
 \begin{equation}
 \begin{split}
     & \rec(X) = \{x\in \R^d\mid \text{ there is some } a\in X \text{ such that }a+tx\in X \text{ for all }t>0\}, \\
     & T_\tau(X)=\bigcup_{\sigma\supset \tau }\{t(x-y)+y\in \R^d\mid y\in \tau,x\in \sigma ,t>0\}.
 \end{split}  
 \end{equation}
 For a tropical hypersurface $X$, both $\rec(X)$ and $T_\tau(X)$ are weighted. The weight of a cell $\sigma$ is the sum of the weights of all cells $\widetilde{\sigma}_1,...,\widetilde{\sigma}_n$ such that $\sigma = \rec(\widetilde{\sigma}_i)$ for all $i$. The weight of a cell $\sigma$ is the weight of the corresponding cell in $X$. See \Cref{fig:recession} for an example.

\begin{figure}[H]
    \centering
    \includegraphics[width=4in]{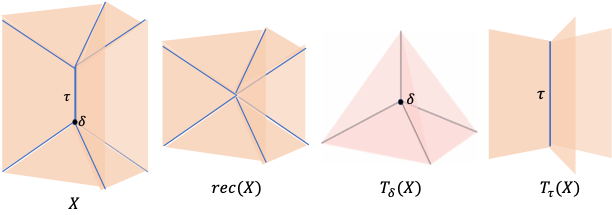}
    \caption{A tropical hypersurface $X$, its recession fan $\rec(X)$, and the tangent fans at $\delta$ and $\tau$.}
    \label{fig:recession}
\end{figure}

\subsection{Tropical rational functions and virtual polytopes}\label{sec:TRF}

 We defined the weight function on $\V(g)$ for a tropical polynomial in \Cref{sec:trop-poly}. The weight function on $\V(\vp)$ for a tropical rational function $\vp$ is defined in the same manner with one more feature: the weight doesn't have to be positive. Suppose a $(d-1)$-face $\sigma$ separate two linear regions $R_1$ and $R_2$. Let $u_1$ and $u_2$ be the gradients of $\vp$ on $R_1$ and $R_2$, respectively. Then 
 \begin{equation}
     w_\vp(\sigma)\begin{cases}
         ||u_1-u_2||, & \text{ if $\vp$ is locally convex along $\sigma$,}\\
          -||u_1-u_2||, & \text{ if $\vp$ is locally concave along $\sigma$.}
     \end{cases}
 \end{equation}
 $w_\vp$ satisfies the same balancing condition as specified in \eqref{eq:balancing-condition}.
 
 Recall that for a tropical rational function $\vp$, $\V(\vp)$ can be decomposed into its positive part $\V(\vp)_+$ and negative part $\V(\vp)_-$, where
\begin{equation}
    \begin{split}
        \V(\vp)_+ = \text{the closure in }\R^d\text{ of }\{x\in \R^d\mid \vp\text{ is locally strictly convex in some neighborhood of }x\},  \\
        \V(\vp)_- = \text{the closure in }\R^d\text{ of }\{x\in \R^d\mid \vp\text{ is locally strictly concave in some neighborhood of }x\}. 
    \end{split}
\end{equation}

 As weighted polyhedral complexes, we may write $\V(\vp)=\V(\vp)_+-\V(\vp)_-$, in light of the operation described in \Cref{sec:operations-on-hyper}. We call this the \textit{Jordan decomposition} of $\V(\vp)$. This name is motivated by the following. For $\vp:\R\to\R$, $\V(\vp)$ is a finite set of points. Each point has a weight indicating the degree of convexity/concavity of $\vp$ at that point. Regarding $\V(\vp)$ as a weighted sum of the indicator functions of all these points, $\V(\vp)$ defines a signed measure on $\R$, and $\V(\vp)=\V(\vp)_+-\V(\vp)_-$ is exactly the Jordan decomposition of $\V(\vp)$ as a signed measure into the difference of two positive measures. Examples of Jordan decompositions are given in \Cref{fig:subtraction}. 

Recall the notion of balancing given in \Cref{def:balancing}: $\V(\vp)_+$ and $\V(\vp)_-$ are in general not tropical hypersurfaces, and a balancing for $\V(\vp)_+$ is a tropical hypersurface that contains $\V(\vp)_+$ as a subcomplex. The following lemma is crucial. It implies that by constructing a balancing for $\V(\vp)_+$, one actually obtains a representation for $\vp$.

\begin{lemma}\label{lem:balancing-gives-rep}
    If $X$ is a balancing of $\V(\vp)_+$, then $Y=(X-\V(\vp)_+)+\V(\vp)_-$ is a tropical hypersurface that balances $\V(\vp)_-$. Therefore, for each balancing $X$ one has a representation $\V(\vp)=X-Y$. 
\end{lemma}
\begin{proof}
  First, $X-\V(\vp)_++\V(\vp)_-=X-(\V(\vp)_+-\V(\vp)_-)=X-\V(\vp)$. Since $X$ and $\V(\vp)$ are both balanced, $X-\V(\vp)$ is balanced. From the definition of a balancing, $X-\V(\vp)_+$ has nonnegative weight. Since $\V(\vp)_-$ also has nonnegative weight, $X-\V(\vp)$ has nonnegative weight. Therefore, by \Cref{prop:equivalance} part (c), it is a tropical hypersurface. Finally, relative to $\V(\vp)_-$, $X-\V(\vp)_++\V(\vp)_-$ satisfies the properties \ref{itm:one} and $\ref{itm:two}$ in the definition of balancing (\Cref{def:balancing}).
\end{proof}
\begin{figure}[H]
    \centering
    \includegraphics[width=5.5in]{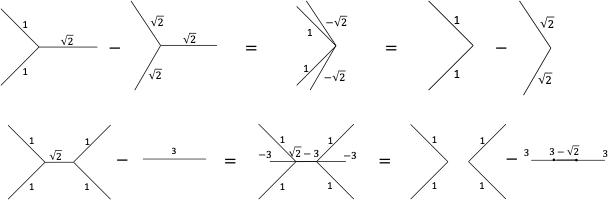}
 \caption{Examples of subtracting two tropical hypersurfaces. The rightmost expression in each equation is the corresponding Jordan decomposition.}
    \label{fig:subtraction}
    
\end{figure}

For a general TRF $\vp$, one cannot associate a Newton polytope if it's not convex. However, one can assign a \textit{virtual polytope} to $\vp$. A virtual polytope is an equivalence class of formal differences of polytopes $P-Q$, where the equivalence relation is defined as
\begin{equation}
    P_1-Q_1\sim P_2-Q_2 \text{ if } P_1+Q_2 = P_2+Q_1.
\end{equation}
It is well-known that the above is indeed an equivalence relation. A virtual polytope have \textit{virtual faces}: recall that faces of polytopes are the set of maximizers of linear functionals. For a virtual polytope $P-Q$, if a linear functional $\ell$ is maximized by the face $F_1$ of $P$ and the face $F_2$ of $Q$, then the virtual face of $P-Q$ corresponding to $\ell$ is the virtual polytope $F_1-F_2$. Virtual polytopes have been studied in other contexts, and they play an important role in quasidifferential calculus (see, for instance, \cite{pallaschke2000minimal}). If $\vp=g\oslash h$, let $\Delta^\uparrow(\vp)=\Delta^\uparrow(g)-\Delta^\uparrow(h)$ and $\Delta(\vp)=\Delta(g)-\Delta(h)$. Choosing a different  representation $g'\oslash h'$ yields $\Delta(g')-\Delta(h')$, which represents the same equivalence class. Therefore, it still makes sense to use the symbol $\Delta^\uparrow(\vp)$ and $\Delta(\vp)$ to denote the virtual polytopes associated with $\vp$. The correspondence in \eqref{eqn:trop-poly-operation-correspondence} extends to tropical rational functions,
\begin{equation}\label{eqn:TRF-operation-correspondence}
     \begin{split}
         \V(\vp^{\odot\alpha}\oslash \psi^{\odot \beta}) = \alpha\V(\vp) - \beta\V(\psi), \\
    \Delta(\vp^{\odot\alpha}\oslash \psi^{\odot \beta}) = \alpha\Delta(\vp) - \beta\Delta(\psi), \\
    \Delta^\uparrow(\vp^{\odot\alpha}\oslash \psi^{\odot \beta}) = \alpha\Delta^\uparrow(\vp) - \beta\Delta^\uparrow(\psi).
     \end{split}
 \end{equation}
 As is in the case of tropical polynomials, for a conewise PL function $\vp$, a virtual $k$-face of $\Delta(\vp)$ is dual to a $(d-k)$-face of $\V(\vp)$. The following results extend \Cref{prop:equivalance} to TRFs.

{\color{red}

}

\begin{prop}\label{prop:equivalence-for-TRF}
    Let $\vp$ and $\psi:\R^d\to \R$ be two reduced TRFs.
    \begin{enumerate}
        \item[(a)] $\V(\vp)=\V(\psi)$ as weighted polyhedral complexes if and only if $\vp$ and $\psi$ differ by a linear function.
        \item[(b)] $\vp=\psi$ if and only if $\Delta^\uparrow(\vp)=\Delta^\uparrow(\psi)$.
        
        \item[(c)] For a pure $(d-1)$-dimensional polyhedral complex $X\subset \R^d$ with weight function $w_X$, there is some TRF $\vp:\R^d\to \R$ such that $\V(\vp)=X$ if and only if $X$ is balanced. 
    \end{enumerate}
\end{prop}
\begin{proof}
The proof for (a) is the same as the the proof for \Cref{prop:equivalance}(a). For (b), note that by \eqref{eqn:TRF-operation-correspondence}, $\vp-\psi=0$ if and only if $\Delta^\uparrow(\vp\oslash\psi) = \{0\}$ if and only if $\Delta^\uparrow(\vp)=\Delta^\uparrow(\psi)$. For (c), given any balanced polyhedral complex $X$, we need to produce a TRF $\vp$ such that $\V(\vp)=X$. Using the same argument in the proof of \Cref{lem:balancing-gives-rep}, we only need to find one balancing for $X_+$, the subcomplex of $X$ with positive weight. For each $d-1$-cell $\sigma$ of $X_+$, let $H_\sigma$ be the hyperplane spanned by $\sigma$ that inherits the weight of $\sigma$. Let $Y$ be the hyperplane arrangement consisting of all such hyperplanes. Namely,
\begin{equation}
    Y=\sum_\sigma H_\sigma,
\end{equation}
 where the sum is the sum of the hyperplanes as tropical hypersurfaces. Then it's clear that $Y$ is a balancing for $X_+$.
\end{proof}
\begin{figure}[H]
    \centering
    \includegraphics[width=4in]{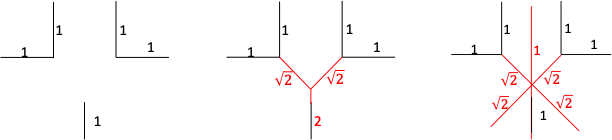}
    \caption{An unbalanced polyhedral complex (left) and two different balancings (middle and right).}
    \label{fig:two-different-balancing}
\end{figure}
In the proof of \Cref{prop:equivalence-for-TRF}(c), we constructed a balancing that consists of the hyperplanes spanned by all the $(d-1)$-faces of $\V(\vp)_+$.
This is a hyperplane arrangement. We call it the \textit{associated arrangement} with $\V(\vp)_+$. An example of the associated arrangement is given in \Cref{fig:canonical-arrangement}.

\begin{figure}
    \centering
    \includegraphics[width=4in]{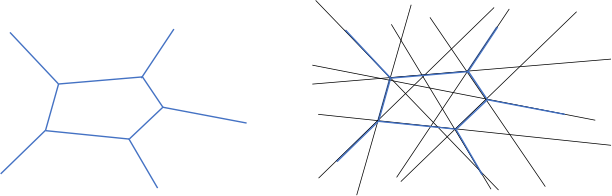}
    \caption{The arrangement (right) associated with a tropical hypersurface (left).}
    \label{fig:canonical-arrangement}
\end{figure}

We end this section with the following observation. For tropical polynomials, the (lifted) Newton polytope and the tropical hypersurface are in one-to-one correspondence (up to a linear function) and there isn't a choice involved. For TRFs, the bend loci has concrete geometry, however, the corresponding polytope constructions are only an equivalence classes.

\section{Complexity and arrangements of tropical hypersurfaces}\label{sec:complexity-of-functions} 

This section is devoted to proving \Cref{thm:lower-bound}. We first observe the following distinct behavior of monomial complexity and factorization complexity.

\begin{prop}\label{prop:two-notions-agree-in-1-d}
For tropical polynomials in one variable, the monomial complexity and the factorization complexity are equal for any non-trivial factorization. There exist tropical polynomials $g,h$ in two variables with irreducible factorizations $g=g_1\odot \cdots \odot g_n$ and $h=h_1\odot \cdots \odot h_m$, such that $\mComp(g)<\mComp(h)$ whereas $\fComp(g_1,...,g_n)>\fComp(h_1,...,h_m)$. 
\end{prop}

\begin{proof}
It is well known that tropical polynomials in one variable factor uniquely into linear factors. The number of tropical roots is the number of linear pieces minus 1, and the factorization complexity is the number of tropical roots plus 1. Hence, for any tropical function $g$ in one variable, $\fComp(g)=\mComp(g)$ for the irreducible factorization and the result follows for arbitrary non-trivial factorizations. To see how the above fails for higher dimensions, consider the following example. Let 
\begin{gather*}
    g(x,y)=(1\odot x^{\odot 1}\odot y^{\odot -1})\oplus (1\odot  y^{\odot -2}) \oplus (1\odot x^{\odot -1}\odot y^{\odot -1})\oplus 1 \oplus y, \\
    h(x,y) = (x^{\odot -1}\odot y^{\odot -1}\oplus 0)\odot ((-2)\odot x^{\odot 1}\odot y^{\odot -1}\oplus 0) \odot (y^{\odot 1}\oplus 0) = h_1\odot h_2\odot h_3
\end{gather*}
Note that $g,h_1,h_2,h_3$ are all irreducible, and $\mComp(g)=5 <\mComp(h)=6$, whereas $\fComp(g)=5>\fComp(h_1,h_2,h_3)=4$. \end{proof}

\Cref{fig:two-complexitys-differ} shows $\V(g)$ and $\V(h)$ as in the proof of \Cref{prop:two-notions-agree-in-1-d} and the dual subdivision of their Newton polytopes. The subdivision defined by $g$ has fewer vertices, but it takes more room to store in a computer since it exhibits less regularity, compared to the zonotope on the right.

\begin{figure}[H]
    \centering
    \includegraphics[width=2.3in]{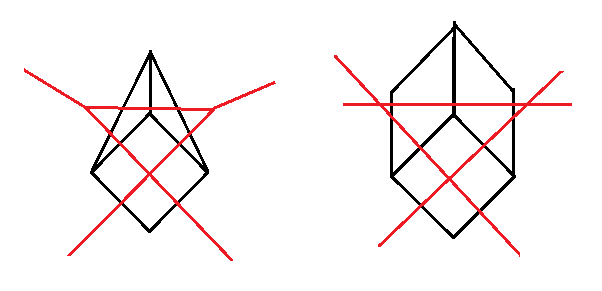}
    \caption{$\V(g)$, $\V(h)$ as in the proof of \Cref{prop:two-notions-agree-in-1-d} and the dual subdivision of the corresponding Newton polytopes.}
    \label{fig:two-complexitys-differ}
\end{figure}

The idea for proving \Cref{thm:lower-bound} is the following. First, we establish the relation between the monomial complexity and the factorization complexity by appealing to the Euler-Poincar\'{e} relation of tropical hypersurfaces and the inclusion-exclusion property of the Euler characteristic. Then, we control the topology of intersections of generic tropical hypersurfaces and bound their Euler characteristics. 

\begin{thm}[Region counting formula for arrangements of tropical hypersurfaces]\label{thm:vertex-counting}
Let $g_1,...,g_m:\R^d\to \R$ be tropical polynomials. Then
\begin{equation}
    \mComp(g_1\odot \cdots \odot g_m) = \fComp(g_1, \cdots, g_m)+\sum_{k=2}^m(-1)^{k+d}\sum_{S\in \binom{[m]}{k}}\chi(\bigcap_{i\in S}\V(g_i)).
\end{equation}
\end{thm}

\begin{proof}
By the Euler-Poincar\'{e} relation,
\begin{equation}\label{eq:counting0}
    \chi(\V(g_1)\cup \cdots \cup \V(g_m))+(-1)^d\mComp(g_1\cup \cdots \cup g_m)=(-1)^d.
\end{equation}
Expand $\chi(\V(g_1)\cup \cdots \cup \V(g_m))$ using \Cref{lem:inclusion-exclusion}
\begin{equation}\label{eq:counting1}
    \chi(\V(g_1)\cup \cdots \cup \V(g_m))= \sum_{i=1}^m \chi(\V(g_i))-\sum_{k=2}^m(-1)^k \sum_{S\in \binom{[m]}{k}}\chi(\bigcap_{i\in S}\V(g_i)).
\end{equation}
Apply Euler-Poincar\'{e} again,
\begin{equation}\label{eq:counting2}
    \sum_{i=1}^m \chi(\V(g_i))= m(-1)^d - (-1)^d\sum_{i=1}^m \mComp(g_i) = (-1)^d- \fComp(g_1, \cdots , g_m).
\end{equation}
Replacing the Euler characteristic in \eqref{eq:counting0} with $\fComp$ using \eqref{eq:counting1} and \eqref{eq:counting2} gives the result.
\end{proof}

 Recall that $\mComp(g_1\odot \cdots \odot g_m)$ is the number of regions the tropical hypersurfaces $\V(g_1)\cup\cdots \cup \V(g_m)$ cut $\R^d$ into (See \Cref{tab:meaning-of-complexity}). One should compare \Cref{thm:vertex-counting} to Theorem 5.5 in \cite{montufar2021sharp} (See \Cref{ex:compare-montufar}). There, the authors express the number of regions in an arrangement of tropical hypersurfaces as the sum of the M\"{o}bius function of each cell weighted by its Euler characteristic, an idea based on Zaslavsky's topological dissection theory. The two formulas are equivalent for tropical hypersurfaces, and one can derive one from the other by using Rota's Crosscut Theorem \cite[Theorem 3]{rota1964foundations}. \cref{thm:vertex-counting} is a regrouping of the terms in the formula provided by \cite{montufar2021sharp}. Any information about the complete intersection gives information about the number of regions. However, computing tropical intersections is in general hard, which has been studied previously in\cite{adiprasito2020note,adiprasito2014filtered,bertrand2007euler}.

\begin{ex}\label{ex:compare-montufar}
The following is an example in \cite{montufar2021sharp}. Their picture shows an arrangement of tropical curves in $\R^2$ and its intersection poset.

\begin{center}
    \includegraphics[width=3.6in]{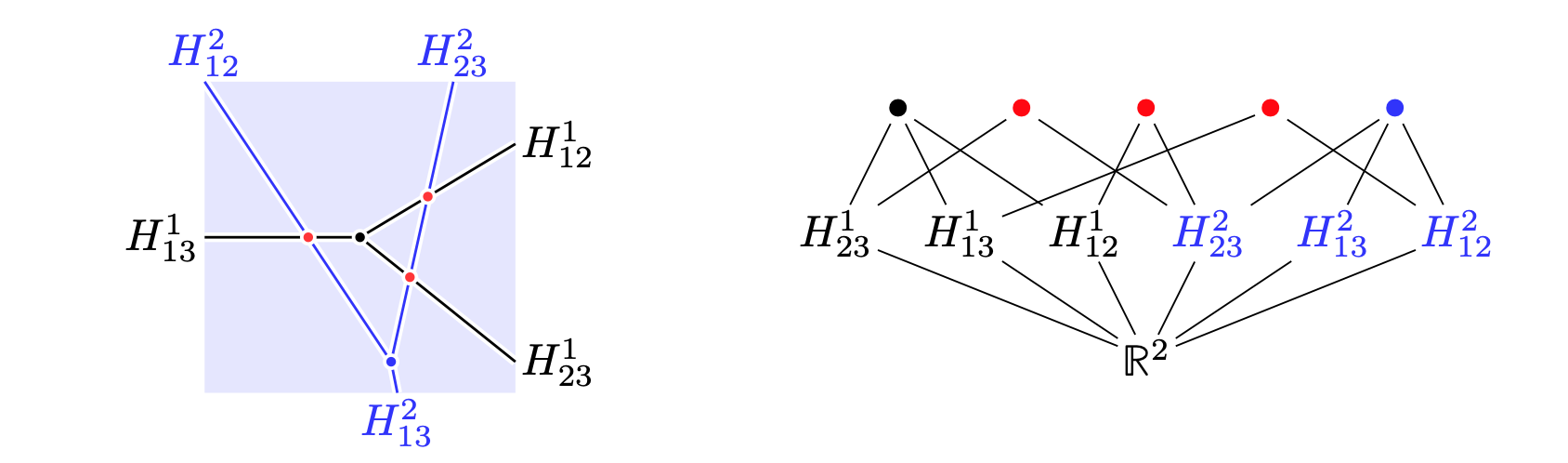}
\end{center}
Each tropical curve divides $\R^2$ into three regions. Their intersection is three points, having Euler characteristic 3. Theorem 1 then tells us the arrangement cuts $\R^2$ into $3+3-2+1+3=8$ regions. Using the formula given in \cite{montufar2021sharp}, we need to compute the M\"{o}bius function and the Euler characteristic of the intersections of the cells. $\mu(\blackdot)=\mu(\bluedot)=2$, $\mu(\reddot)=1$, $\mu(H_{ab}^i)=-1$, $\mu(\R^2)=1$. This yields 
\[r=1+0(-1-1-1-1-1-1)+(-1)^0(2+2+1+1+1)=8\]
regions.
\end{ex}

In machine learning, the number of linear regions a neural network has is a measure of its expressivity. Given a tropical rational function $\vp=g\oslash h$, one can associate a surrogate function $g\odot h$. The number of linear regions is bounded from above by the number of linear regions of $g\odot h$, which can be computed using \Cref{thm:vertex-counting}. This implies the following.

\begin{cor}\label{cor:rational-linear-region}
Let $\vp=g\oslash h$ be a tropical rational function  on $\R^d$. Then
\begin{equation}
    \text{ linear regions of }\vp\leq \mComp(g) + \mComp(h) + (-1)^{d+2}\chi(\V(g)\cap\V(h)). 
\end{equation}
\end{cor}

Now we prepare another ingredient for the proof of \Cref{thm:lower-bound}. Using Stratified Morse Theory, Adiprasito shows that if all the tropical hypersurfaces in a complete intersection are pointed, meaning that they divide the ambient space $\R^d$ into pointed polyhedra, then their intersection is homotopy Cohen-Macaulay (\cite{adiprasito2020note}, Theorem 1.1), which implies that the one-point compactification of the intersection is a wedge of spheres of the same dimension. Moreover, the bounded part of the intersection, meaning the subcomplex consisting of bounded polyhedra, is also homotopy equivalent to a wedge of spheres. We adopt his idea in our setting and obtain the following lemma. Since the proof is a rather long topological argument, we include it in full detail in the appendix for this paper \cite{wang2023topology}. 

\begin{lemma}\label{lem:topology-of-intersection} Let $X\subset \R^d$ be a complete intersection of generic tropical hypersurfaces.
The one-point compactification and the bounded part of $X$ are homotopy equivalent to wedges of $n$-spheres.
\end{lemma}

\begin{proof}
    By the main result in the appendix \cite{wang2023topology}, $X\cup\{pt\}$ and the bounded part of $X$ are $(n-1)$-connected. In the following, let $Y$ be either $X\cup \{pt\}$ or the bounded part of $X$. Let $\widetilde{H}_k(Y)$ be the $k$-th reduced homology group of $Y$ with integer coefficients and $\pi_k(Y)$ be its $k$-th homotopy group. By Hurewicz Theorem (see, e.g. \cite[Theorem 4.32]{MR1867354}),
    \begin{equation}\label{eq:hurewicz}
        \widetilde{H}_k(Y) = \begin{cases}
            0, & k< n, \\
            \pi_n(Y), & k=n.
        \end{cases}
    \end{equation}
    Since $Y$ is an $n$-dimensional finite CW-complex, $\widetilde{H}_n(Y)$ is freely generated by finitely many elements. By \eqref{eq:hurewicz}, $\pi_n(Y)$ is also freely generated by finitely many elements, say, $\gamma_1,...,\gamma_m$. Those generators give a map $\gamma_1\vee \cdots \vee \gamma_m$ from a wedge of $m$ $n$-spheres to $Y$, which induces isomorphisms on homology groups. By Whitehead's Theorem (see, e.g. \cite[Corollary 4.33]{MR1867354}), that map is a homotopy equivalence.
\end{proof}

\begin{lemma}\label{lem:euler-characteristic-recession}
Let $X\subset \R^d$ be a complete intersection of generic tropical hypersurfaces. Then
$\chi(\rec(X))\leq \chi(X)$ if $n$ is even, and $\chi(\rec(X))\geq \chi(X)$ if $n$ is odd.
\end{lemma}

\begin{proof}
Let $Y$ be the subcomplex of $X$ consisting all the bounded polyhedra. Let $\rec(X)$ be the recession fan of $X$ (see \Cref{fig:recession} and the definition above it). Then $\chi(X\bs Y)=\chi(\rec(X))-1=\chi(X)-\chi(Y)$. If $Y$ is contractible, i.e. homotopy equivalent to an empty wedge of spheres, then $\chi(Y)=1$ and $\chi(\rec(X))=\chi(X)$. If $Y$ is homotopy equivalent to a non-empty wedge of $n$-spheres, then $\chi(Y)\geq 2$ if $n$ is even, and $\chi(Y)\leq 0$ if $n$ is odd, which implies the desired inequality.
\end{proof}

\begin{lemma}\label{lemma:euler-bound}
Let $X\subset \R^d$ be a complete intersection of generic tropical hypersurfaces. Then $\chi(X)\geq 1$ if $n$ is even, and $\chi(X)\leq -1$ if $n$ is odd. The equality holds if and only if $X$ is an affine subspace.
\end{lemma}

\begin{proof}
By \Cref{lem:topology-of-intersection}, the weak inequalities hold. We only need to show that strict inequalities hold when $X$ is not an affine subspace. For $n=0$, $X$ consists of finitely many points. $X$ is not an affine subspace if and only if $X$ contains more than one point, so the claim is obvious. Suppose the claim holds for $n\leq k$ and that $X$ has dimension $k+1$. Consider $\rec(X)$. If $X$ is not an affine subspace, neither is $\rec(X)$. By \Cref{lem:euler-characteristic-recession}, when $k+1$ is even, $\chi(\rec(X))\leq \chi(X)$, whereas when $k+1$ is odd, $\chi(\rec(X))\geq \chi(X)$. Therefore, it suffices to prove the claim for fans. Now assume $X$ is a fan. Consider the following construction. Intersect $X$ with a generic hyperplane $H$ that goes through the base point of $X$. We decompose $X$ into three disjoint sub-complexes $X\cap H$, $X\cap H^{>0}$ and $X\cap H^{<0}$, where $H^{>0}$ and $H^{<0}$ are the half-spaces defined by $H$. Let $H_+$ be any hyperplane parallel to $H$ that lives in $H^{>0}$ and $H_-$ be any hyperplane parallel to $H$ that lives in $H^{<0}$. Let $X_1=X\cap H_+$ and $X_2=X\cap H_-$. Then $X\cap H^{>0}$ is homeomorphic to $X_1\times \R$ and $X\cap H^{<0}$ is homeomorphic to $X_2
\times \R$. Therefore, $\chi(X\cap H^{>0})=-\chi(X_1),\chi(X\cap H^{<0})=-\chi(X_2)$. Moreover, $X\cap H=\rec(X_1)=\rec(X_2)$. Putting everything together, we get
\[\chi(X)=\chi(X\cap H^{>0})+\chi(X\cap H^{<0})+\chi(X\cap H)=-\chi(X_1)-\chi(X_2)+\chi(\rec(X_1)).\]
Note that $X_1,X_2$ and $\rec(X_1)$ are generic intersections of dimension $k$. Moreover, at least one of $X_1$ and $X_2$ is not an affine subspace. If $k$ is even, then 
\[-\chi(X_1)-\chi(X_2)+\chi(\rec(X_1)) \leq \min\{-\chi(X_1),-\chi(X_2)\} < -1,\]
while if $k$ is odd,
\[-\chi(X_1)-\chi(X_2)+\chi(\rec(X_1))\geq \max\{-\chi(X_1),-\chi(X_2)\}> 1.\]
This completes the proof.
\end{proof}

\begin{proof}[Proof of \Cref{thm:lower-bound}]
Since all the tropical hypersurfaces are in general position, any $k$ of them have empty intersection if $k\geq d+1$. By \Cref{lemma:euler-bound}, if $S$ has cardinality $k$, then
\[\chi(\bigcap_{i\in S} \V(g_i))\leq -1,\]
if $d-k$ is odd, and 
\[\chi(\bigcap_{i\in S} \V(g_i))\geq 1,\]
if $d-k$ is even. Therefore, we have
\[(-1)^{k+d}\chi(\bigcap_{i\in S} \V(g_i))\geq 1,\]
which yields the lower bound by \Cref{thm:vertex-counting}. By \Cref{lemma:euler-bound}, the equality holds if and only if all the intersections $\cap_{i\in S}\V(g_i)$ are affine subspaces, which happens if and only if $g_i$'s are all binomials.
\end{proof}

 \Cref{thm:lower-bound} shows that when the number of factors grows, the factorization complexity grows linearly while the monomial complexity grows in 
 $O(m^d)$. This also means with respect to factorization complexity, a minimal representation of $g$ must be an irreducible factorization of $g$, whenever all factors have generic parameters. However, we note that the converse is not true, as different irreducible factorizations may have different complexities, as is shown in the following example. Let $g=(x\oplus y)\odot(x\oplus 0)\odot (y\oplus 0)=(x\oplus y\oplus 0 )\odot (x\odot y \oplus x\oplus y)$. The first factorization has complexity 4, while the second has complexity 5. 
 
 \Cref{thm:lower-bound} has several interpretations. First, it says the number of regions in an arrangement of generic tropical hypersurfaces is no less than the number of regions in an arrangement of hyperplanes, and is strictly larger than the latter unless all the hypersurfaces are ordinary hyperplanes. Second, recall that the number of regions in an arrangement of tropical hypersurfaces corresponds to the number of vertices in the regular mixed subdivision of the Minkowski sum of the corresponding Newton polytopes. Therefore, \Cref{thm:lower-bound} is a lower bound in that setting. A special case is when the mixed subdivision is trivial, namely, when it doesn't really subdivide the Minkowski sum of the polytopes. This corresponds to arrangements of balanced polyhedral fans which are based at the same point. We have
 
 \begin{thm}\label{thm:lower-bound-polytope}
 Let $P_1,...,P_m\subset \R^d$ be polytopes in general position. Then 
 \begin{equation}
     \mComp(P_1+\cdots + P_m)\geq \sum_{i=1}^m\mComp(P_i)+2\sum_{k=0}^{d-1}\binom{m-1}{k}-2m.
 \end{equation}
 The equality holds if and only if all $P_i$'s are line segments.
 \end{thm}

 \begin{proof}
 Consider the corresponding tropical hypersurfaces of those polytopes. Consider the intersection of $k$ of them. If $k\leq d$, the argument is the same as in \Cref{thm:lower-bound}. If $k>d$, then their intersection is a single point. Applying \Cref{lemma:euler-bound} and \Cref{thm:vertex-counting}, one gets
\[\mComp(P_1+\cdots + P_m)\geq  \sum_{i=1}^m\mComp(P_i)+\sum_{k=2}^d\binom{m}{k}+\sum_{k=d+1}^m(-1)^{k+d}\binom{m}{k}=\sum_{i=1}^m\mComp(P_i)+2\sum_{k=0}^{d-1}\binom{m-1}{k}-2m.\]
All of their intersections are affine subspaces if and only if all $P_i$'s are line segments, which gives the condition for equality by \Cref{thm:lower-bound}.
\end{proof}

 \begin{rmk}
 In \Cref{thm:lower-bound-polytope}, if all $P_i$'s are polytopes of positive dimension, then $\mComp(P_i)\geq 2$ for all $i$, which implies
 \[\mComp(P_1+\cdots +P_m)\geq 2\sum_{k=0}^{d-1}\binom{m-1}{k}. \]
 This gives another proof for the Lower Bound Theorem by Adiprasito \cite[Corollary 8.2]{adiprasito2017lefschetz}, which says a Minkowski sum of polytopes of positive dimension in general position has at least the same number of vertices as the Minkowski sum of the same number of line segments in general position. When all the polytopes $P_1,...,P_m$ are full-dimensional, the precise lower bound obtained in \cite{adiprasito2017lefschetz} is
 \[\mComp(P_1+\cdots+P_m)\geq \sum_{i=1}^m\mComp(P_i)+\binom{d+m-1}{d-1}-d-m+1.\]
 For a fixed $d$, this bound is better than \Cref{thm:lower-bound-polytope} when $m$ is small. When $m$ is large, note that 
 \[\binom{d+m-1}{d-1}-d-m+1\sim \frac{m^{d-1}}{(d-1)!},\qquad 2\sum_{k=1}^{d-1}\binom{m-1}{k}-2m \sim \frac{2m^{d-1}}{(d-1)!}.\]
 Therefore, \Cref{thm:lower-bound-polytope} is a slightly better bound asymptotically.
 \end{rmk}
 
 \begin{rmk}
 In \cite{tseran2021expected} Theorem 8, Tseran and Mont\'{u}far obtain a lower bound for the number of activation regions for a maxout network. Their proof is based on the result in \cite{adiprasito2017lefschetz} and a lower bound on the number of strictly upper vertices of a Minkowski sum \cite[Theorem 6.9]{montufar2021sharp}. Their result can be re-obtained by applying \Cref{thm:lower-bound}. We also note that \Cref{thm:lower-bound} as well as \Cref{thm:lower-bound-polytope} is not true in general if $g_i$'s don't have generic parameters. For instance, let $h_1(x,y)=x\oplus y\oplus 0$ and $h_2(x,y)=x^{\odot 2}\oplus y\oplus 0$. Then $h_1\odot h_2 = x^{\odot 3}\oplus y^{\odot 2} \oplus x^{\odot 2}\odot y\oplus 0$. Note that $\fComp(h_1, h_2)=5$, but $\mComp(h_1\odot h_2)=4$.  
 \end{rmk}

\section{Minimal representation and minimal balancing}\label{sec:minimal-rep}

In this section we prove \Cref{prop:fan-unique} and \Cref{thm:piecewise-conewise}. Recall that a tropical polynomial $g$ is called reduced if it doesn't have any redundant monomials. We call a representation $\vp=g\oslash h$ \textit{reduced} if both $g$ and $h$ are reduced. We call a factorization $\vp=(g_1\odot \cdots \odot g_n)\oslash (h_1\odot \cdots \odot h_m)$ reduced if all of the factors are. The \textit{monomial complexity} (resp. \textit{factorization complexity}) of $\vp$ is the tuple $(\mComp(g),\mComp(h))$ (resp. $(\fComp(g_1,...,g_n),\fComp(h_1,...,h_m))$).
 We note that a minimal representation is necessarily reduced, whereas a reduced representation is not necessarily minimal. 

\begin{ex}\label{ex:reduced-not-minimal}
    Recall the function in \Cref{ex:rational_function}
 \[
    x\oplus y\oplus 0 - x\oplus y = x\odot y\oplus  x\oplus y\oplus 0 - x\odot y \oplus x\oplus y.
\]
Both sides are reduced and represent the same PL function, but the right hand side has monomial complexity $(4,3)$ while the left hand side has monomial complexity $(3,2)$. If we factor $x\odot y\oplus  x\oplus y\oplus 0 $ into $(x\oplus 0)\odot (y\oplus 0)$, then the right hand side has factorization complexity $(3,3)$. In whichever case, the left hand side is a shorter representation. In fact, it is the unique minimal representation w.r.t. either of the monomial complexity or the factorization complexity. In general, minimal representations are not unique, but they are when $d=1$.
\end{ex}

\begin{prop}\label{prop:minimal-unique-1-d}
Every tropical rational function in one variable has a unique minimal representation up to a linear function. Namely, if $\vp=g_1\oslash h_1$ and $\vp=g_2\oslash h_2$ are two minimal representations, then $g_1\oslash g_2=h_1\oslash h_2$ is a linear function. Moreover, a representation for $\vp$ is minimal if and only if its reduced.
\end{prop}

    \begin{proof}
    By \Cref{prop:two-notions-agree-in-1-d}, monomial complexity and factorization complexity agree when $d=1$. Suppose $\vp=g\oslash h$ is reduced. This means that $\V(\vp)=\V(g)\sqcup\V(h)$, $\V(g)$ is the positive part of $\V(\vp)$ while $\V(h)$ is the negative part of $\V(\vp)$. Therefore, $\mComp(g)=|\V(g)|+1$ and $\mComp(h)=|\V(h)|+1$. Note that this is the lower bound for the complexity of any representation. Hence, $g\oslash h$ must already be minimal. For uniqueness, suppose $g_1\oslash h_1$ and $g_2\oslash h_2$ are two minimal representations. Then $\V(g_1)=\V(g_2)$ as weighted polyhedral complex. This means $g_1\oslash g_2$ must be a linear function. \end{proof}
    
    In other words, \Cref{prop:minimal-unique-1-d} says \ref{itm:feature-2} holds for PL functions in one variable. Therefore, we can write down the minimal representation by looking at the zeros and poles: suppose we are given $\vp:\R \to \R$, knowing $\V(\vp)=\{x_1,...,x_n\}\sqcup \{y_1,...,y_m\}$, where $\vp$ is strictly convex at $x_i$'s and strictly concave at $y_i$'s. Let 
    
    \[\psi(x)=\frac{(x\oplus x_1)^{\odot a_1}\odot\cdots \odot (x\oplus x_m)^{\odot a_m} }{(x\oplus y_1)^{\odot b_1}\odot\cdots \odot (x\oplus y_n)^{\odot b_n}}, \]
    where $a_i$'s and $b_i$'s are the absolute value of the slope change of $\vp$ around each non-differentiable point. Then $\psi$ and $\vp$ differ by a linear function, which can be determined by computing two values of $\vp$.

   Due to the failuer of \ref{itm:feature-2}, minimal representation problems in higher dimensions are a lot more subtle. We make two reductions. The first reduction is to consider the corresponding minimal balancing problem. More precisely, a minimal representation problem takes the following form,
   \begin{equation}
        \begin{split}
           & \min_{g,h} \quad (\mComp(g),\mComp(h)), \\
            & s.t.\quad \vp=g\oslash h.
        \end{split}
    \end{equation}
Since finding a representation is the same as finding a balancing for $\V(\vp)_+$ (\Cref{lem:balancing-gives-rep}), one can instead consider the following optimization,      \begin{equation}
\begin{split}
    & \min_X \quad  \mComp(X), \\
    & s.t.\quad  X \text{ is a balancing of } \V(\vp)_+.    
\end{split}
\end{equation}
   
   The is the \textit{minimal balancing problem} for $\V(\vp)_+$. Similarly, one can also consider minimal balancing for $\V(\vp)_-$. A minimal balancing problem only has one value to optimize, thus is more feasible than solving a minimal representation directly. \Cref{fig:lattice} expresses the relations between the two problems schematically. Note that not all minimal representations come from minimal balancing.

   \begin{figure}[H]
      \centering
        \includegraphics[width=3.5in]{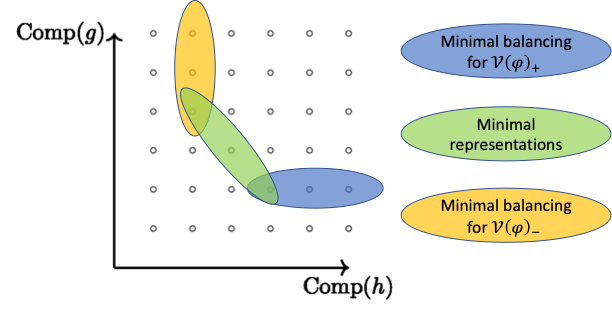}
        \caption{Relation between the minimal balancing problem and the minimal representation problem. Comp is either mComp or fComp.}
      \label{fig:lattice}
  \end{figure}

   Our second reduction is to study the minimal representation problem for conewise linear functions and 
   the corresponding balancing problem for polyhedral fans. Polyhedral complexes are locally polyhedral fans.
   If finding the minimal balancing for the whole complex is hard, one may resort to finding minimal balancing for local polyhedral fans first. This is supported by the following proposition, which will imply \Cref{thm:piecewise-conewise}. We remind the reader that for a tropical hypersurface $X$, $T_\sigma (X)$ is its tangent fan at the face $\sigma$ (see \Cref{fig:recession} and the definition above it).

 \begin{prop}\label{prop:decomp-into-fans}
 Let $X$ be a tropical hypersurface. Then
 \begin{equation} \label{eqn:decomp-into-fans}
     X=\sum_{ \substack {\tau \text{ a compact} \\ \text{face of $X$}}}(-1)^{\dim \sigma}T_\sigma(X). 
 \end{equation}
\end{prop}

\begin{proof}
Recall the conic decomposition for polytopes (a.k.a. Brianchon-Gram formula. See \cite{haase2005polar}),
    \begin{equation}
        \mathbbm{1}_P = \sum_{F \text{ a face of }P} (-1)^{\dim F}  \mathbbm{1}_{T_FP},
    \end{equation}
    where $T_FP$ is the tangent cone of $P$ at a face $F$ and $\mathbbm{1}$ denotes the indicator function. By applying the Minkowski-Weyl theorem, this formula can be extended to the following formula for not necessarily compact polyhedra,
\begin{equation}\label{eq:MW-noncompact}
        \mathbbm{1}_P = \sum_{F \text{ a compact face of }P} (-1)^{\dim F}  \mathbbm{1}_{T_FP}.
    \end{equation}
By \eqref{eq:MW-noncompact}, we deduce that for any $(d-1)$-face $\sigma$ of $X$,
\begin{equation}\label{eq:conic-applied}
    w(\sigma)\mathbbm{1}_\sigma = \sum_{ \substack {\tau \text{ a compact} \\ \text{face of $\sigma$}}} (-1)^{\dim \tau} w(\sigma)\mathbbm{1}_{T_\tau \sigma}.
\end{equation}
Summing up \eqref{eq:conic-applied} over all $(d-1)$-faces of $X$, we get
\begin{equation}\label{eq:sum-over-all-cells}
    \sum_{\sigma} w(\sigma)\mathbbm{1}_\sigma = \sum_{\sigma}\sum_{ \substack {\tau \text{ a compact} \\ \text{face of $\sigma$}}} (-1)^{\dim \tau} w(\sigma)\mathbbm{1}_{T_\tau \sigma} = \sum_{ \substack {\tau \text{ a compact} \\ \text{face of $X$}}}(-1)^{\dim \tau} \sum_{\sigma\supset \tau} w(\sigma)\mathbbm{1}_{T_\tau \sigma}.
\end{equation}
The second equality comes from exchanging the order of two summations. The whole equality \eqref{eq:sum-over-all-cells} is a reformulation of \eqref{eqn:decomp-into-fans} in terms of indicator functions, where the left-hand side is equal to $X$ and the right-hand side is equal to the alternating sum of tangent fans.
\end{proof}

\begin{proof}[Proof of \Cref{thm:piecewise-conewise}]
    It suffices to prove the statement for any tropical polynomial, since any piecewise-linear function is a difference of two tropical polynomials. By \Cref{prop:decomp-into-fans}, $\V(g)$ is a linear combination of tropical hypersurfaces that are fans for any tropical polynomial $g$. By \Cref{prop:equivalance} and \eqref{eqn:TRF-operation-correspondence}, $g$ is a linear combination of conewise linear functions. 
\end{proof}

\begin{rmk}
 \cref{prop:decomp-into-fans} suggests how to decompose a PL function into linear combination of conewise linear functions. Consider the convex PL function $g(x,y)=\max\{x,-x,y+1,-y\}$. The tropical hypersurface $X=\V(g)$ and its decomposition into linear combination of tangent fans are shown in \Cref{fig:decomp-into-conewise}. $X$ has three compact faces: two 0-dimensional faces $\tau_1$ and $\tau_2$ and one 1-dimensional face $\sigma$. Around $\tau_1$, $g$ agrees with the function $\max\{-x,y+1,-y\}$; around $\sigma$, it agrees with $\max\{y+1,-y\}$; around $\tau_2$, it agrees with $\max\{x,y+1,-y\}$. \cref{prop:decomp-into-fans} says 
 \begin{equation}
     g(x,y) = \max\{-x,y+1,-y\} - \max\{y+1,-y\} + \max\{x,y+1,-y\}.
 \end{equation}
\end{rmk}

\begin{figure}[H]
    \centering    \includegraphics[width=5.5in]{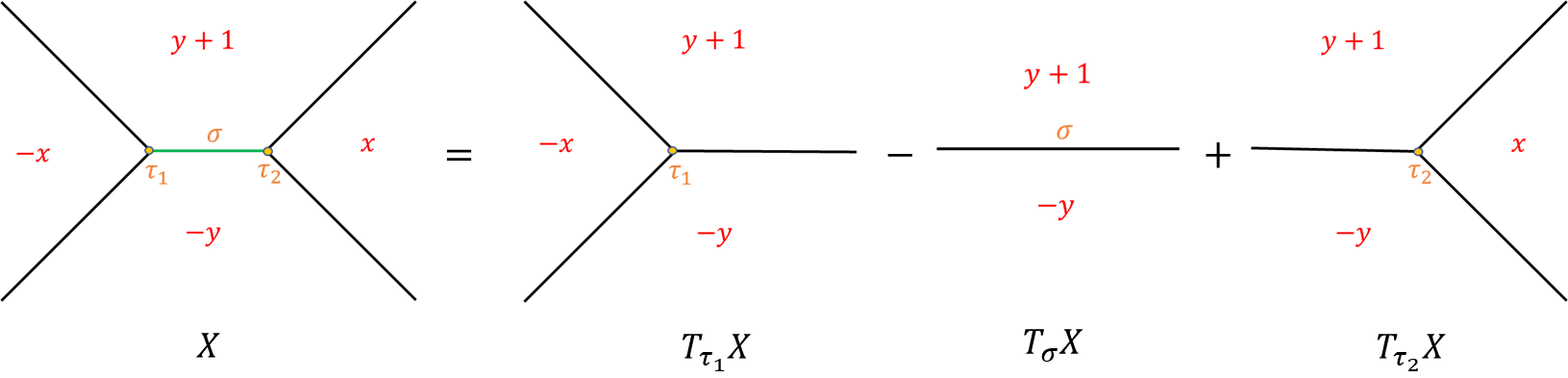} 
    \caption{Decomposition of a convex PL function into a linear combination of conewise linear functions.}
    \label{fig:decomp-into-conewise}
\end{figure}

\subsection{The story for monomial complexity}\label{sec:monomial-comp}

In this subsection, we study the minimal representation/balancing problem w.r.t. monomial complexity. We first observe the case when $\V(\vp)_+$ and $\V(\vp)_-$ are tropical hypersurfaces.

\begin{prop}\label{prop:when-positive-balanced}
       If $\V(\vp)_+$ is balanced, then then there is a unique minimal representation $g\oslash h$ where $\V(g)=\V(\vp)_+$ and $\V(h)=\V(\vp)_-$.
   \end{prop}
\begin{proof}
    Suppose $\V(g')$ balances $\V(\vp)_+$ and $\V(g')\neq \V(g)$. Then $\mComp(g')\geq \mComp(g)$. Moreover, $g'=g\odot r$ for some tropical polynomial $r$ that is not a linear function. The corresponding balancing for $\V(\vp)_-$ must be $\V(h)+\V(r)$. Note that $\V(r)$ and $\V(h)$ doesn't have the same support, so $\mComp(\V(h)+\V(r))>\mComp(\V(h))$, which means $g'\oslash(h\odot r)$ is not a minimal representation.
\end{proof}

Recall the notion of \textit{canonical balancing} defined in the introduction:
Suppose $\vp:\R^2\to \R$ is a conewise linear function. Then $\V(\vp)$ is a collection of rays with the same origin $\tau$. Suppose $\V(\vp)_+$ has rays $\sigma_1,...,\sigma_s$ with weights $w_1,...,w_s$, respectively, while $\V(\vp)_-$ has rays $\sigma_{s+1},...,\sigma_m$ with weights $-w_{s+1},...,-w_m$, respectively. For each $\sigma_k$, let $v_k$ be the unit vector generating $\sigma_k$. The balancing condition at $\tau$ says
\begin{equation}
\sum_{k=1}^s w_kv_k    =\sum_{k=s+1}^mw_kv_k.
\end{equation}

Let $\sigma_{m+1}$ be ray originated from $\tau$ generated by the vector $-\sum_{k=1}^s w_kv_k$. Note that $\sigma_{m+1}$ is empty if $\sum_{k=1}^s w_kv_k=0$. Set the weight on $\sigma_{m+1}$ to be $w_{m+1}=||\sum_{k=1}^s w_kv_k||$. Let $X=\V(\vp)_+\cup \sigma_{m+1}$ and $Y=\V(\vp)_- \cup \sigma_{m+1}$ with the specified weights. By construction, $X$ and $Y$ are balanced. We note two features of this construction. \begin{itemize}
    \item It is symmetric. Since one can choose to balance $\V(\vp)_+$ or $\V(\vp)_-$, a priori the representation obtained from the balancing depends on this choice. However, the representation obtained from above doesn't depend on the choice. In other words, the balancing described above can be simultaneously constructed for $\V(\vp)_+$ and $\V(\vp)_-$. Besides \Cref{ex:rational_function}, \Cref{fig:canonical-balancing} is another example.
    \item If $\V(\vp)_+$ is already balanced, this construction doesn't do anything.
\end{itemize}
In contrast, the associated arrangement defined in \Cref{sec:TRF} rarely satisfies the above properties. These two properties will be made precise at the end of this section. They are desirable properties for a balancing construction to be canonical. We therefore call the above construction that is now specific to PL functions on $\R^2$ the \textit{canonical balancing}. The resulting representation is called the \textit{canonical representation}. Besides satisfying the above properties, \Cref{prop:fan-unique} says that the canonical representation is minimal w.r.t. $\mComp$.

\begin{figure}[H]
    \centering
    \includegraphics[width=5in]{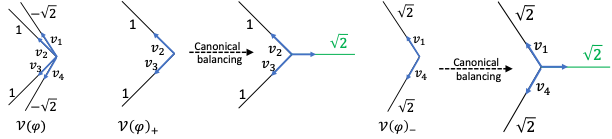}
    \caption{Canonical balancing can be constructed simultaneously for $\V(\vp)_+$ and $\V(\vp)_-$, giving the canonical representation.}
    \label{fig:canonical-balancing}
\end{figure}

\begin{proof}[Proof of \Cref{prop:fan-unique}]

    Let $g\oslash h$ be the canonical representation of $\vp$. 
   Suppose $\V(\vp)
_+$ consists of $m_1$ rays and $\V(\vp)_-$ consists of $m_2$ rays. There are three possibilities.
\begin{enumerate}
        \item $\mComp(g\oslash h)=(m_1,m_2)$. This happens if and only if $\V(\vp)_+$ is already balanced, in which case the canonical representation is the representation $\V(\vp)_+-\V(\vp)_-$. By \Cref{prop:when-positive-balanced}, it is minimal.

        \item $\mComp(g\oslash h) = (m_1+1,m_2)$. This happens if the new ray in the canonical balancing $\V(g)$ is in the support of $\V(\vp)_-$. In this case, $\V(\vp)_+$ is unbalanced. We need to show there isn't a representation $g'\oslash h'$ such that $\mComp(g'\oslash h')=(m_1,m_2)$. If such a $g'\oslash h'$ exists, then $\V(g')$ and $\V(\vp)_+$ have the same support, and $\V(h')$ and $\V(\vp)_-$ have the same support. This can only happen when $\V(\vp)_+$ is already balanced, a contradiction. By symmetry, this also works for the case $\mComp(g\oslash h)=(m_1,m_2+1)$.

        \item $\mComp(g\oslash h)= (m_1+1,m_2+1)$. It suffices to show that there isn't a representation $g'\oslash h'$ such that $\mComp(g'\oslash h')=(m_1,m_2+1)$. Suppose $\mComp(g')=m_1$.
        Since $\V(g')$ is not the canonical balancing for $\V(\vp)_+$, $\V(g')-\V(\vp)_+$ consists of at least two rays that are in the support of $\V(\vp)_+$. These two rays are not in the support of $\V(\vp)_-$, so $\V(h')=\V(g')-\V(\vp)_++\V(\vp)_-$ has monomial complexity at least $m_2+2$. By symmetry, we also showed that there isn't $g'\oslash h'$ such that $\mComp(g'\oslash h')=(m_1+1,m_2)$.
    \end{enumerate}
    
   This completes the proof.    
\end{proof}

For general $\vp:\R^2\to \R$, the representation obtained above may not be the unique minimal one, and the canonical balancing for $\V(\vp)_+$ is not necessarily minimal, as is already suggested in the proof. When $\V(\vp)_+$ or $\V(\vp)_-$ contains a balanced subcomplex, one can construct a balancing $\V(g')$ that has the same support as $\V(\vp)_+$ by increasing the weights for some rays of $\V(\vp)_+$.

Traveling from the land of tropical hypersurfaces to the land of polytopes, we get the following two corollaries, which will also be helpful for higher dimensional generalization of the canonical balancing. 
\begin{cor}\label{cor:polytope-version}
\begin{enumerate}
    \item If two polytopes $P,Q$ are transverse to each other, then $P-Q$ is the unique minimal representative for the virtual polytope it represents w.r.t. $\mComp$. 
    \item Any virtual polygon has a canonical representative $P-Q$, such that $P\subset P'$ and $Q\subset Q'$ (up to translation) whenever  $P'-Q'=P-Q$. Moreover, $P-Q$ is minimal w.r.t. $\mComp$.
\end{enumerate}
\end{cor}

\begin{proof}
    The first statement is just a rephrase of \Cref{prop:when-positive-balanced} in terms of polytopes. The second statement requires a little bit more work. Recall that a polygon and its normal fan satisfies the following relation.
Take the direction $v_i$ of each ray of the normal fan $X$, each having length the weight of that ray. Label all the vectors $v_1,...,v_n$ in the counterclockwise direction, meaning that if ray $\gamma$ has directional vector $v_i$ then the next ray in the counterclockwise direction is $v_{i+1}$. Consider the sequence of points $0,v_1,v_1+v_2,v_1+v_2+v_3,...,v_1+\cdots +v_{n-1}$. The balancing condition says $\sum v_i=0$. In other words, the vectors $v_1,...,v_n$ can enclose a polygon whose vertices are the points given above. Rotating this polygon counterclockwise by 90 degrees gives $P$ (see \Cref{fig:polytope-version}).

 If one starts with $\vp$ that corresponds to a virtual polygon $P-Q$, one can repeat the above construction with $\V(\vp)_+$. However, we can only get part of a polygon this way, if $\V(\vp)_+$ is unbalanced. One way to make it a polygon is to simply connect the starting point and the end point, and this is exactly the canonical balancing. It is clear that any other way of completing the edges into a polygon will give a polygon that contains $P$.
\end{proof} 

\begin{figure}[H]
    \centering
    \includegraphics[width=4.5in]{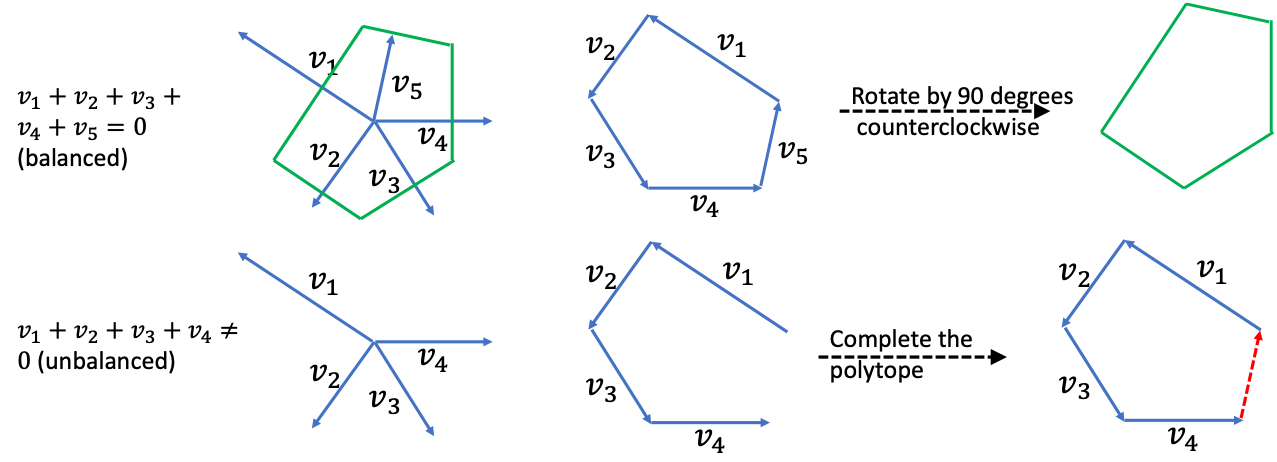}
    \caption{Relation between a polygon and its normal fan, and the polytopal description of the canonical balancing.}
    \label{fig:polytope-version}
\end{figure}

   Now we construct a potential generalization of canonical balancing to higher dimensions. Let $\vp:\R^d\to\R$ be a conewise linear function represented by some virtual polytope. Let $\tau$ be a $(d-2)$-cell of $\V(\vp)$. Suppose $\sigma_1,...,\sigma_n$ are the $(d-1)$-cells containing $\tau$ as a face. Then $\tau$ corresponds to a virtual 2-face (see the definition of virtual faces in \Cref{sec:TRF}) of $P-Q$ and $\sigma_1,...,\sigma_n$ correspond to its virtual edges. We know how to construct the canonical representative for this virtual 2-face, and the idea is to construct the canonical representatives for all 2-faces and take their convex hull. This procedure is broken down into the following steps.

Step 1: choose a $(d-2)$-cell $\tau_1$ from $\V(\vp)_+$. Suppose $\sigma_1,...,\sigma_n$ are the $(d-1)$-cells of $\V(\vp)_+$ containing $\tau_1$. The half-spaces spanned by $\sigma_1,...,\sigma_n$ form a 1-dimensional fan in $\R^d/\R \tau$. Let $P_1$ be the polygon that corresponds to the canonical balancing for this fan.
    
    Step 2: construct all such polygons $P_1,...,P_m$ for all the $(d-2)$-cells in a way that is compatible with the cell structure of $\V(\vp)_+$. Namely, if two $(d-2)$-cells $\tau_i$ and $\tau_j$ both are faces of $\sigma$, then $P_i$ and $P_j$ share an edge that is dual to $\sigma$. 
    
    Step 3: take the convex hull $P$ of $P_1,...,P_m$. Let $\V(g)$ be the tropical hypersurface corresponding to $P$. The \textit{canonical balancing} for $\V(\vp)_+$ is $\V(g)$.

\Cref{fig:algorithm} is an example of the above algorithm in dimension 3. The input is an unbalanced fan $X$, which is part of the normal fan of the tetrahedron, and can be through of as $\V(\vp)_+$ for
\begin{equation}
    \vp(x,y,z) = x^{\odot -1} \oplus x^{\odot 1} \oplus y^{\odot 2} \oplus y^{\odot 1}\odot z^{\odot -1} - x^{\odot -1} \oplus x^{\odot 1} \oplus y^{\odot 2}.
\end{equation}
It consists of three rays and three 2-cells. It is unbalanced around the three rays, so we perform the canonical balancing at each ray $\tau_1,\tau_2,\tau_3$.
 For each $\tau_i$, we get a 2-face $P_i$. Gluing those 2-faces together and take the convex hull, we get a tetrahedron, whose normal fan balances $X$.
    
    \begin{figure}[H]
        \centering
        \includegraphics[width=5.5in]{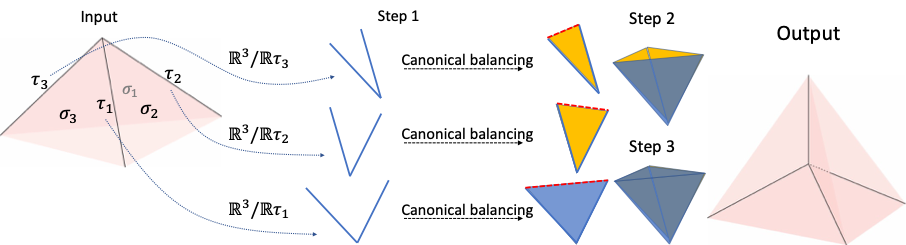}
        \caption{Given an unbalanced fan, the algorithm outputs the normal fan of the tetrahedron.}
        \label{fig:algorithm}
    \end{figure}

The fact that one can define a canonical representative for virtual polytopes in $\R^2$ should not be taken for granted. It might be special for dimension 2. It relies on the fact that in dimension 2, Minkowski addition is equivalent to \textit{Blaschke addition} \cite[Section 15.3]{grunbaum1967convex}. By a theorem by Minkowski, a $d$-dimensional polytope $P$ in $\R^d$ is equivalent to a collection of vectors that sum up to zero (called an equilibrated system \cite[Page 332]{grunbaum1967convex}). Each vector is normal to a facet of $P$, and the length of each vector is the $(d-1)$-dimensional volume of the corresponding facet. Since superposing two equilibrated systems produces another equilibrated system, this defines an addition on the set of polytopes. This addition, in general, is drastically different from Minkowski sum, but happens to be the same as the Minkowski sum for polygons. This equivalence is implicit in the construction of the canonical representation. It is reminiscent of the fact that, for algebraic surfaces, curves and divisors agree. Therefore, we hesitate to call the above construction a valid generalization. The main question to address is the following. Regard the above construction as a map $\mathcal{B}$ where
\begin{equation}
    \mathcal{B}: \{\V(\vp)_+\mid \vp \text{ a PL function}\} \to \{\text{representations of }\vp\mid \vp \text{ a PL function}\},\quad \V(\vp)_+\mapsto g\oslash h.
\end{equation}
For $\mathcal{B}$ to live up to being ``canonical", we'd like it to satisfy the following two properties. 
\begin{itemize}
    \item (Normality) If $\V(\vp)_+$ is already balanced, then $\mathcal{B}(\V(\vp)_+)=g\oslash h$ where $\V(g)=\V(\vp)_+$. Namely, $\mathcal{B}$ doesn't do anything to those $\vp$ whose $\V(\vp)_+$ is already balanced.
    \item (Symmetry) $\mathcal{B}(\V(\vp)_+) = \mathcal{B}(\V(\vp)_-)$. Namely, balancing $\V(\vp)_+$ and balancing $\V(\vp)_-$ via $\mathcal{B}$ gives the same representation. In other words, the balancing defined by $\mathcal{B}$ is obtained for $\V(\vp)_+$ and $\V(\vp)_-$ simultaneously.
\end{itemize}
In $\R^2$, $\mathcal{B}$ is both normal and symmetric, whereas the associated arrangement is neither. In higher dimensions, $\mathcal{B}$ is normal, and we don't know whether or not it is symmetric.

\subsection{The story for factorization complexity}

In this subsection, we study the minimal balancing problem w.r.t. factorization complexity. The notion of factorization complexity has a local nature: on one hand, it's the sum of the contribution of each factors. On the other hand, the minimal balancing problem may also be approached by locally by computing the minimal balancing of subcomplexes and then combine all the local solutions. One feature of factorization complexity is that global optimality implies local optimality, made precise by the following proposition.

\begin{prop}\label{prop:properties-of-factorization-len}
Let $X=\V(\vp)_+$ for some TRF $\vp$. If $g_1\odot\cdots \odot g_m$ minimizes factorization complexity for $X$, then any factor $g_{n_1}\odot\cdots\odot g_{n_k}$ minimizes factorization complexity for $X\cap \V(g_{n_1}\odot\cdots \odot g_{n_k})$.
\end{prop}

\begin{proof}

Suppose for a contradiction that some factor, which by relabelling the indices we may assume to be $g_1\odot\cdots \odot g_k$, does not balance $X\cap \V(g_1\odot\cdots\odot g_k)$ minimally. Then there is some other balancing $f_1\odot \cdots \odot f_s$ having smaller factorization complexity. Note that $f_1\odot\cdots\odot f_s \odot g_{k+1}\odot \cdots \odot g_m$ is also a balancing of $X$. We compute 
\begin{align*}
    \fComp(f_1,\cdots, f_s , g_{k+1}, \cdots \, g_m) & =\fComp(f_1,\cdots,f_s )+\fComp(g_{k+1}, \cdots , g_m)-1 \\
    & <\fComp(g_1,\cdots, g_k )+\fComp(g_{k+1},\cdots , g_m)-1 \\
    & = \fComp(g_1,\cdots,g_m).
\end{align*}
This contradicts to the minimality of $g_1\odot\cdots \odot g_m$.
\end{proof}

To see how monomial complexity fails to have the above property, consider the example shown in \Cref{fig:total-example}. Suppose the blue fan $X_2$ is already balanced and the black fan $X_1$ is unbalanced. By \Cref{thm:vertex-counting}, the number of intersection points contribute positively to the number of regions. To minimize the monomial complexity, one should create as few intersection points as possible. If we find the minimal balancing for $X_1$ and $X_2$ separately, we get the middle picture, which is not minimal for $X_1+X_2$. On the other hand, if we find the minimal balancing for $X_1+X_2$, we get $\overline{X_1}+\overline{X_2}$ in the third picture, which is not minimal for $X_1$ alone.

\begin{figure}[H]
    \centering
        \includegraphics[width=4in]{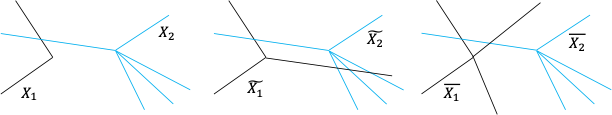}
    \caption{Balancing the union of $X_1+X_2$ (left) in two ways. Local minimization yields the picture in the middle, while global minimization yields the picture on the right.}
    \label{fig:total-example}
\end{figure}

So far, we've only dealt with a single unbalanced fan. The following proposition concerns a collection of unbalanced fans in general positions, in which case local optimality implies global optimality, providing a partial converse to \Cref{prop:properties-of-factorization-len}.

\begin{prop}\label{prop:generic-fans}
Suppose $\vp_1,...,\vp_m\subset \R^2$ are conewise linear functions whose parameters are drawn from some continuous probability distribution, and that $X_1=\V(\vp_1)_+,...,X_m=\V(\vp_m)_+$. Suppose
$\V_1,...,\V_m$ balance $X_1,...,X_m$ minimally w.r.t. $\fComp$, respectively. Then $\V_1+\cdots + \V_m$ balances $X_1+\cdots+ X_m$ minimally, with respect to $\fComp$.
\end{prop}

\begin{proof}
 Suppose $X_k$ contains $r_k$ rays. Suppose $\W_1\cup\cdots\cup \W_n$ balances $X_1\cup \cdots\cup X_m$.
Then $\rec(\W_1+\cdots + \W_n)=\rec(\W_1)+\cdots +\rec(\W_n)$ balances $\rec(X_1+ \cdots + X_m)$. Since each $\rec(\W_i)$ balances $\rec(\W_i) \cap \rec(X_1+ \cdots + X_m)$, the number of rays in $\rec(\W_i)$ is at least the number of rays in $\rec(\W_i) \cap \rec(X_1+ \cdots + X_m)$ plus 1. Therefore, 
\[\fComp(\rec(\W_1\+\cdots+\W_n))\geq \sum_{k=1}^mr_k+1=\fComp(\V_1+\cdots+\V_m).\]
Since $\fComp(\W_1+\cdots+ \W_n)\geq \fComp(\rec(\W_1+\cdots+\W_n))$, we have
\[\fComp(\W_1+\cdots+ \W_n)\geq \fComp(\V_1+\cdots+\V_m).\]
Therefore, $\V_1+\cdots+ \V_m$ is a minimal balancing for $X_1+\cdots+ X_m$.
\end{proof}

Now we turn to another perspective of the balancing problem. Instead of seeking for a balancing of minimal complexity, one may be satisfied with one that has small complexity. Recall that the associated arrangement $\A_X$ is the union of all hyperplanes spanned by the $(d-1)$-cells of $X$. \Cref{prop:lower-bound-balancing} shows that, in $\R^2$, $\A_X$ has small factorization complexity.

\begin{proof}[Proof of \Cref{prop:lower-bound-balancing}]
We first consider the case when the minimal balancing of $X$ is $\V(g)$ for some irreducible tropical polynomial $g$. In that case, the number of hyperplanes in $\A_X$ is no more than the 1-cells in $\V(g)$.  The number of 1-cells in $\V(g)$ is the number of edges in the regular subdivision of $\Delta(g)$ induced by $\Delta^\uparrow(g)$ and $\mComp(g)$ is the number of vertices. Let $l$ be the number of 2-cells in this subdivision. By the Euler-Poincar\'{e} relation, \[\fComp(\A_X)= \mComp(g) + l.\]
Since every 2-cell has at least three 1-cells in its boundary, and every 1-cell is in the boundary of at most two 2-cells, we have
\[2(\fComp(\A_X)-1)\geq 3l.\]
From there we get
\[\fComp(\A_X)\leq 3\mComp(g)-2.\]
Suppose the minimal balancing of $X$ is given by a factorization $g_1\odot\cdots\odot g_m$. Each $g_i$, by \Cref{prop:properties-of-factorization-len}, balances $\V(g_i)\cap X$ minimally. Set $X_i=\V(g_i)\cap X$. Then for each $i$ we have
\[\fComp(\A_{X_i})\leq 3\mComp(g_i)-2.\]
Therefore,

\[ \frac{\fComp(\A_X)}{\fComp(g_1,\cdots, g_m)} \leq\frac{\sum_{i}\fComp(\A_{X_i})-(m-1)}{\sum_{i}\mComp(g_i)-(m-1)} \leq \frac{\sum_{i}3\mComp(g_1)-3m+1}{\sum_{i}\mComp(g_i)-(m-1)}\leq 3.  \]
\end{proof}

\section{Conclusion}\label{sec:conclusion}

We studied the complexity of representing PL functions as tropical rational functions and the representations with minimal complexity. We expect the content to be of interest for both discrete geometers and statisticians, so we conclude the paper with open questions in both directions.

In \Cref{sec:monomial-comp} we defined a potential generalization of the canonical balancing to higher dimensions. We don't know if this generalization is really canonical in higher dimensions, in the sense that it is both normal and symmetric. We are pessimistic, since many facts about polytopes are special in dimension 2. However, we see no obvious obstruction for $\mathcal{B}$ to be symmetric. 

\begin{opq}
    Is the balancing $\mathcal{B}$ defined in \Cref{sec:monomial-comp} symmetric? Does it give minimal representation w.r.t. $\mComp$?
\end{opq}

Another question not clear to us is whether or not tropical polynomials have unique irreducible factorizations with minimal factorization complexity. Tropical polynomials don't have unique irreducible factorizations. However, do we get the uniqueness if we only consider those with minimal factorization complexity? Note that the answer is negative if we drop the word "irreducible".

\begin{opq} Do tropical polynomials have unique irreducible factorization minimal w.r.t. the factorization complexity?
\end{opq}

In general, we wonder how to turn
 the results in this paper into concrete statistical applications. In particular, since \Cref{thm:lower-bound} suggests that the number of parameters can be reduced using factorization, does this give any advantage in practice? According to \Cref{thm:piecewise-conewise}, any piecewise-linear function is a linear combination of local conewise linear function. Is a linear spline model consisting of local conewise linear models better than a global piecewise-linear model? Does it have any implication for neural network initialization?

\begin{opq}
    Does factorization improve statistical performance?
\end{opq}

\begin{opq}
    What is the statistical implication of \Cref{thm:piecewise-conewise}?
\end{opq}

\printbibliography

\end{document}